\documentclass[10pt]{article}
\usepackage{graphicx} 
\usepackage{amsfonts,amsthm}
\usepackage{eucal}
\usepackage{amssymb,epsfig,latexsym}
\usepackage{graphicx,amsmath,amssymb,latexsym,psfrag}
\usepackage{graphics,graphicx}
\usepackage[dvipsnames,svgnames]{xcolor}
\usepackage[
backend=biber,
style=numeric,
sorting=nyvt, maxnames=99
]{biblatex}
\usepackage[colorlinks=true, urlcolor=blue,linkcolor=blue, citecolor=blue]{hyperref}
\newtheorem{Theorem}{Theorem}[section]
\newtheorem{Proposition}[Theorem]{Proposition}

\newtheorem{Lemma}[Theorem]{Lemma}
\newtheorem{Remark}[Theorem]{Remark}

\newtheorem{Assumption}[Theorem]{Assumption}

\usepackage{mathtools}

\newcommand{\R}{{\mathbb R}}

\newcommand{\N}{{\mathbb N}}

\newcommand{\PP}{{\mathbb P}}
\newcommand{\E}{{\mathbb E}}

\renewcommand{\d}{{\operatorname d}}

\renewcommand{\d}{{\operatorname d}}

\newcommand{\cA}{\mathcal{A}}

\newcommand{\cM}{\mathcal{M}}

\numberwithin{equation}{section}

\title{Markov approximation for controlled Hawkes Jump-Diffusions with general kernels}
\author{Mahmoud Khabou\footnote{Imperial College London (NeST Project), United Kingdom, m.khabou@imperial.ac.uk} \quad Mehdi Talbi\footnote{Laboratoire de Probabilités, Statistiques et Modélisation, Université Paris-Cité, France, talbi@lpsm.paris}} 
\date{\today}

\begin{document}

\maketitle

\begin{abstract}
    We present a Markov approximation for jump-diffusions whose jump part consists in a Hawkes process with intensity driven by a general (possibly non-monotone) kernel. Under minimal integrability conditions, the kernel can be approximated by a linear combination of exponential functions. This implies that Hawkes jump-diffusions can be approximated with Markov jump-diffusions. We illustrate the usefulness of this approximation by applying it to a class of stochastic control problems. 
\end{abstract}

\vspace{3mm}
\section{Introduction and preliminary results}\label{sec:preliminaries}
\subsection{Introduction}

Many observed quantities of interest have a discontinuous component: think, for instance, of spikes in a network of neurons (\citeauthor{galves} \cite{galves}) or of the sudden drop in stock prices (\citeauthor{MERTON1976125} \cite{MERTON1976125}). This explains the popularity of jump-diffusion stochastic differential equations \textit{(SDE)} in many applications such as gene networks (\citeauthor{gene} \cite{gene}), coupled chemical reactions (\citeauthor{chemical} \cite{chemical}) and biological stochastic systems (\citeauthor{bressloff} \cite{bressloff}).\\

The popularity of standard jump-diffusion SDEs also stems from their Markov property. We refer the reader to \citeauthor{bally} \cite{bally} for results on jump-diffusions with state-dependent intensity from a Markov process perspective. This property entails that the control of such dynamics can be studied within the standard dynamic programming framework: see, for example, \citeauthor{benazzoli} \cite{benazzoli} for the mean-field stochastic control of an interbank network, \citeauthor{rajabi2025optimal} \cite{rajabi2025optimal}, \citeauthor{harrison2005stochastic} \cite{harrison2005stochastic} or \citeauthor{Costa2024.10.02.616330} \cite{Costa2024.10.02.616330} for the stochastic control of neuronal networks, and \citeauthor{dewitte} \cite{dewitte} for portfolio optimization. See also \citeauthor{oksendal2007applied} \cite{oksendal2007applied} for a general overview of this subject.\\ 

However, due to the independence of their increments, the standard jump processes (Poisson or L\'evy) fail to capture the clustering phenomenon that is observed in many time series (financial, neurological, geological etc.). This motivated Alan G.\ Hawkes to introduce the eponymous Hawkes point process \cite{hawkes}, which allows self and mutual excitation between nodes in a network. This has then been extended to allow for self and mutual inhibition by Br\'emaud and Massouli\'e \cite{bm}.\\

This prompts us to introduce Hawkes jump-diffusions, which can be informally defined as follows. First, let $W=(W_t)_{t\in \R_+}$ be a standard Brownian motion and $N= (N_t)_{t\in \R_+}$ be a counting process, \textit{i.e.} an increasing, piecewise constant process with jumps of size one. The intensity $\lambda =(\lambda_t)_{t\in \R_+}$ is a predictable process measuring the jump-rate of $N$ knowing the past, that is, 
 $$\lambda_t \d t = \mathbb P \left[N_{t+\d t}-N_t=1 \big | \mathcal F_{t-}\right],$$ 
 where $\mathcal F$ is the filtration associated with $N$ and a Brownian motion $W$. We say that $(X,\lambda)$ is a Hawkes jump-diffusion if it solves the stochastic integro-differential equation
 \begin{equation}
     \label{eq:EDS_Intro}
     \begin{cases}
          X_t &= X_0 + \int_0^t \mu(X_s) \d s + \int_0^t\sigma(X_s) \d W_s + \int_0^t\gamma (X_{s-}) \d N_s\\
         \lambda _t &=  \lambda_{\infty} + \int_{_0}^{t-} \phi(t-s) \nu (X_{s-}) \d N_s
     \end{cases}
 \end{equation}
where $\mu, \sigma$ and $\gamma$ are the drift, volatility and jump functions respectively. The function $\phi$ is called the \textit{kernel} and encodes the effect of past jumps on the future. The function $\nu$ allows this impact to also depend on the state of $X$ during past jumps. \\

A case of interest corresponds to kernels of the form $\phi(t)=\eta e^{-\beta t}$. Due to their memorylessness, Hawkes intensities with exponential kernels may be rewritten as Markov SDEs with jumps. Such dynamics have been extensively studied in the literature, including in the context of control problems, where it can be analyzed using the dynamic programming approach: see \textit{e.g.} \citeauthor{callegaro2025stochasticgordonloebmodeloptimal} \cite{callegaro2025stochasticgordonloebmodeloptimal} for applications in cybersecurity, \citeauthor{mazzoran} \cite{mazzoran} for applications in energy markets, Aït-Sahalia and Hurd \cite{ait2015portfolio} in portfolio optimization and Bensoussan, \citeauthor{bergault2024price} \cite{bergault2024price} in market making and Bensoussan and Chevalier-Roignant \cite{bensoussan2024stochastic} for a general overview of controlled Hawkes processes with exponential kernels. \\

However, the exponential kernel is too restrictive and does not capture many of the properties that characterize real world systems. For example, its monotonicity does not allow us to take into account the delayed excitation / inhibition and refractory period characteristic of neuronal networks (see \citeauthor{bielajew1987effect} \cite{bielajew1987effect}). Moreover, an empirical study on cybersecurity data by Bessy-Roland, Boumezoued and Hillairet \cite{Bessy-Roland_Boumezoued_Hillairet_2021} shows that the kernel $\phi(t)=\eta t e^{-\beta t}$ fits the data better than the exponential.
This is why we need to address the case of a general kernel, where the intensity of the Hawkes process corresponds to a stochastic Volterra-integral equation. This class of dynamics has received a strong attention over the past few years, due to their diverse applications such as in medical sciences (see \textit{e.g.} \citeauthor{schmiegel2006self} \cite{schmiegel2006self} and \citeauthor{saeedian2017memory} \cite{saeedian2017memory}) or in finance, in particular in the study of rough volatility models (see \textit{e.g.} \citeauthor{bayer2016pricing} \cite{bayer2016pricing} or \citeauthor{gatheral2018volatility} \cite{gatheral2018volatility}).  \\

The control of Volterra-type dynamics is particularly challenging: 
due to the presence of the kernel $\phi$ in the dynamics \eqref{eq:EDS_Intro}, this problem is in general path-dependent, and therefore cannot be analyzed through the standard dynamic programming approach for Markov processes. Additionally, the intensity of the Hawkes process is not even a semimartingale. This control problem therefore cannot be analyzed using the theory of path-dependent partial differential equations, such as developed by Ekren, Keller, Touzi \& Zhang \cite{ekren2014viscosity}. \\

Various approaches have been considered to study control problems involving Volterra-type dynamics. One rather popular method is to handle the problem through a maximum
principle approach, see \textit{e.g.} \citeauthor{agram2015malliavin} \cite{agram2015malliavin}, \citeauthor{agram2018new} \cite{agram2018new}, \citeauthor{lin2020controlled} \cite{lin2020controlled} and 
\citeauthor{hamaguchi2023maximum} \cite{hamaguchi2023maximum}. We also mention the recent contribution of \citeauthor{cardenas2022existence} \cite{cardenas2022existence}, who search for an optimal control in a relaxed form.   \\

A large number of papers focus on the so called \textit{lifting} approach, which roughly consists in seeing the Volterra-type diffusion as the image of an infinite dimensional Markov process by some linear form, see e.g\ \citeauthor{abi2021linear} \cite{abi2021linear}, \citeauthor{jusselin2021optimal} \cite{jusselin2021optimal},  \citeauthor{di2023lifting} \cite{di2023lifting} or \citeauthor{hamaguchi2023markovian} \cite{hamaguchi2023markovian}.   In a slightly different approach, \citeauthor{viens2019martingale} \cite{viens2019martingale} lifts the state process---typically a fractional Brownian motion---in the Banach space of continuous path, treating the `Volterra time' as a parameter. This approach has been used in several subsequent works, such as the ones of \citeauthor{wang2022path} \cite{wang2022path} and \citeauthor{wang2023linear} \cite{wang2023linear}. 
In a recent work, Possamaï and Talbi \cite{possamai2024optimal} propose to lift the Volterra dynamics in the Sobolev space on $[0,T]$, which is possible whenever the kernels are weakly differentiable. \\ 

Another approach relies on the following observation: whenever the kernel of the Volterra dynamics writes as a linear combination of exponential functions, it can be extended into a multi-dimensional Markovian system. In particular, Hawkes processes with exponential kernels may be rewritten as standard Poisson-types jump processes, whose control problems can be analyzed in a standard way. It is therefore natural to search for exponential approximations of nonexponential kernel. It is well known that, thanks to Bernstein's Theorem, completely monotone kernels are subject to such approximations, see \textit{e.g.} Coutin and Carmona \cite{carmona1998fractional}, Alfonsi and Kebaier \cite{alfonsi2024approximation}, Jusselin \cite{jusselin2021optimal} or Bayer and Breneis \cite{bayer2023markovian}.  \\

In the present paper, we use the fact that the exponential approximation holds for a larger class of nonexponential kernels, under mere integrability conditions. This allows us to approximate path-dependent Hawkes jump-diffusion processes with multidimensional Markov jump diffusions, even when the kernel non-monotonous. This approximation is particularly useful in stochastic control, as it allows to characterize the (approximated) value function with as the solution of a partial differential equation. \\

The paper is structured as follows. After the Introduction, Section \ref{sec:preliminaries} introduces our probabilistic framework and the main assumptions on the considered dynamics. Section \ref{sec:continuity} analyzes the continuity of Hawkes jump-diffusions with respect to the kernel of the intensity process, which is a crucial step towards our general Markov approximation result. Section \ref{sec:main} presents the exponential approximations for integrable kernels and the main result of this paper, which is the Markov approximation of Hawkes-jump diffusions. The approximation results provided are given with their dependence on the time horizon, which is useful for numerical applications. In Section \ref{sec:control}, we show how our approximation can be used to analyze stochastic control problems, and illustrate the usefulness of our method with an application to a portfolio optimization problem. Section \ref{sec:proofs} contains the proofs of our main results and the Appendices \ref{sec:appendix} and \ref{sec:other_proofs} present some useful lemmata and proofs of the remaining results. \\ \\
\textbf{Notations.} Given an integer $d \ge 1$ and  vectors $\boldsymbol{x}=(x^1,\cdots,x^d)$ and $\boldsymbol{y}=(y^1,\cdots,y^d)$, we denote by $\boldsymbol{xy}$ the Hadamard (or component-wise) product $\boldsymbol{xy}=(x^1 y^1,\cdots, x^d y^d)$, and by $\boldsymbol{x} \cdot \boldsymbol{y}$ the scalar product $\boldsymbol{x} \cdot \boldsymbol{y} = \sum_{i=1}^d x^i y^i$. If $E$ is a Borel set, we denote by $\mathcal{B}(E)$ the corresponding Borel algebra. If $(\Omega, \mathbb F, \mathbb P)$ is a filtered probability space, we denote by $\mathbb{H}^2$ the set of $\mathbb{F}$-adapted processes $Z$ such that $\E\left[\int_0^T | Z_t |^2 dt \right] < \infty$, where the process $Z$ takes values in a arbitrary Banach space $(E, |\cdot |)$, and $\E$ denotes the expectation under the probability measure $\mathbb P$. 


\subsection{Preliminary results}

Throughout the paper, we denote by $T > 0$ a fixed time horizon.   Let $m$ be a Borel measure on $\R$ such that $m(\R)=1$. We introduce the space of jump configurations:
$$\hat \Omega^{\Pi}=\Big\{\omega=\sum_{i=1}^n \delta_{(t_i,\theta_i,y_i)}, 0<t_1<t_2<\cdots <t_n,(\theta_i,y_i)\in \R _+ \times \R , n\in \N \cup \{+\infty\} \Big\}.$$
Let $\Omega ^{\Pi}:=\prod _{j=1}^d \hat \Omega ^{\Pi}$
We denote by $\mathbb P^\Pi$ the Poisson measure under which the counting process $\boldsymbol {\Pi}:=(\Pi^1,\cdots,\Pi^d)$ defined as
$$\Pi^i \left([0,t]\times [0,\theta]\times (-\infty,y]\right)(\omega^i):=\omega ^i\left([0,t]\times [0,\theta]\times (-\infty,y]\right),\quad (t,\theta,y)\in \R_+\times \R_+ \times \R $$
for $i=1,\cdots,d$ is a homogeneous Poisson process with intensity measure $\d t \d \theta m(\d y)$. The natural filtration associated with $\Pi$ is denoted by $\mathbb F^\Pi=\left(\mathbb F^\Pi_t\right)_{t\geq 0}$ and is defined by
$$\mathbb F_t^\Pi:=\sigma \left(\Pi(\mathcal T \times D), \mathcal T \subset \mathcal B([0,t]^d),D\in \mathcal B((\R_+\times \R)^d)\right).$$
The third component $y$ corresponds to a random mark associated to the point $(t,\theta)$. This can be seen as the random intensity of a spike in the context of a neuronal network or the amount of a claim in an insurance framework. \\
If the points hold a constant non-zero mark (say $1$ for instance), that is $m(\d t) = \delta_1(\d t)$, we drop the third component and we simply write $\Pi (\d t, \d \theta)$.\\
We let $\boldsymbol W=( \boldsymbol W_t)_{t\geq 0}$ be a standard $k-$ dimensional Brownian motion constructed on a different probability space $\Omega ^W$, associated with the probability measure $\mathbb P^W$ and the natural filtration $\mathbb F^W$. We define $\Omega = \Omega ^{\Pi} \times \Omega ^W$ and $\mathbb P:=\mathbb P^\Pi \otimes \mathbb P^W$, $\E$ to be the expectation with respect to $\mathbb P$ and $\mathbb F= \{ \mathcal{F}_t\}_{\{0 \le t \le T\}} = \mathbb F^\Pi \vee \mathbb F^W$. The conditional expectation is defined by 
$$\E_t [\cdot]=\E [\cdot |\mathbb F_t].$$
Let $(\boldsymbol{\lambda}_{\infty},\boldsymbol{\mu}, \sigma, \boldsymbol{\gamma}, \nu) : [0,T] \times \mathbb{R}^d \times \Omega \longrightarrow \R^d\times  \mathbb{R}^d \times \cM_{d, k}(\mathbb R) \times \mathbb{R}^d \times \mathbb{R}$ and $\phi: [0,T] \longrightarrow \cM _{d,d}(\R)$ be measurable functions. In this paper, we are interested in the following system of stochastic integro-differential equations:
\begin{equation}
\label{eq:EDS}
    \begin{cases}
        X^i_t &= x^i_0 + \int_0^t \mu^i(s,\boldsymbol X_s) \d s + \int_0 ^t \sum_{j=1}^k\sigma^{ij}(s,\boldsymbol X_s) \d W^j_s + \int_0^t \int_{\R_+ \times \R} y\gamma^i(s,\boldsymbol{X}_{s-}) \boldsymbol{1}_{\theta \leq \lambda^i_s} \Pi^i (\d s, \d \theta, \d y)\\
        \lambda^i_t &= \lambda^i_{\infty}(t,\boldsymbol{X}_{t-})+\psi \left(\int_0^{t-} \int_{\R_+ \times \R}\sum_{j=1}^d\phi^{ij} (t-s) b(y) \nu (s,\boldsymbol{X}_{s-})\boldsymbol{1}_{\theta \leq \lambda^j_s} \Pi^j (\d s, \d \theta ,\d y) \right)
    \end{cases}
\end{equation}
where  $\boldsymbol{x}_0=(x_0^1,\cdots,x^d_0)$ is a given initial condition and $\psi: \R \longrightarrow \R_+$ is a non-decreasing jump-rate function. \\
The functional coefficients $(\boldsymbol{\lambda}_{\infty},\boldsymbol{\mu}, \sigma, \boldsymbol{\gamma})$ have a straightforward interpretation in the Markov jump-diffusion framework with state dependent intensity (\textit{cf.} \cite{bally} and the references therein). However the kernel $\phi$ is a novelty in jump-diffusion SDEs. It encodes the effect of the jumps of node $j$ on node $i$: the presence of a jump in node $j$ at an instant $\tau$ increases the integral in the expression of $\lambda^i_t$ by $\phi^{ij}(\tau-t)$ which means that it is more (respectively less) likely to observe a jump in $X^i_t$ if $\phi^{ij}$ is positive (respectively negative). We call such phenomenon cross-excitation (respectively cross-inhibition). \\
Naturally, due the this positive feedback loop in the positive kernel case, there is a risk of ``explosion" if a mutual cross-excitation or a self-excitation has a high intensity or is too persistent in time. To avoid such a phenomenon, we introduce the stability assumption \ref{ass:stability}, which is now standard in the Hawkes process literature. \\
Before proving the existence of the processes solving the SDE \eqref{eq:EDS}, we state the following assumptions:
\begin{Assumption}
\label{ass:Lipschitz}
    For all $\varphi \in \{\mu, \sigma, \gamma, \nu\}$, the process $t \mapsto \varphi(t,x)$ is in $\mathbb{H}^2$ for all $x \in \mathbb{R}^d$, and we have:
    $$|\varphi(t,x)-\varphi(t,y)|\leq C |x-y|,$$
    for some deterministic $C>0$ indepdendent of $t \in [0,T]$. Moreoever, the baseline intensity $\lambda_\infty$ is non-negative and has a Lipschitz constant $L_\lambda$ which is strictly less than one, that is
    $$\sup_{s\in \R_+,(x,y)\in \R^2}\frac{\left | \lambda_{\infty}(s,x)-\lambda_\infty(s,y)\right |}{x-y}=L_\lambda <1.$$
    Furthermore, $\lambda_\infty$, $\phi$, and $\gamma$ are bounded and
    $\int y^2 m(\d y) < + \infty$.
\end{Assumption}
\begin{Assumption}
    \label{ass:stability}
    The jump rate $\psi$ is positive and $L-$Lipschitz, the kernel $\phi$ is in $\mathbb L^1([0,T],\d t)$ and $ \int b(y) m(\d y) < + \infty$. Moreover, $\nu$ is uniformly bounded by 1 and
    $$L\E b(Y) \mathrm{sp}\left(\|\phi\|_1\right)<1,$$
    where $\mathrm{sp}$ is the spectral radius and, for $p \ge 1$, $\|\phi\|_p=\left(\Big[\int_0^T|\phi^{ij}(s)|\d s\Big]^{1/p}\right)_{i,j=1,\cdots,d}$.
\end{Assumption}
\begin{Assumption}
    \label{ass:gronwall}
    One of these two conditions holds:
    \begin{enumerate}
        \item The jump rate $\psi$ is bounded from above.
        \item The functions $\gamma$ and $\nu$ do not depend on $x$.
    \end{enumerate}
\end{Assumption}

Even though the SDE \eqref{eq:EDS} is multivariate, in the remainder of this paper all the results are written in the univariate case to reduce clutter. 
\begin{Proposition}
\label{prop:SDE}
    Let $T>0$ and $x_0 \in \R$. If Assumptions \ref{ass:Lipschitz} and \ref{ass:gronwall} hold, then SDE \eqref{eq:EDS} has a unique solution $(X,\lambda)$ on $[0,T]$ satisfying 
    $$\E \left[\sup_{t\in [0,T]}|X_t| \right] < + \infty.$$
    Additionally, if for $p\geq 2$, $Y$ and $b(Y)$ have finite $p-$th moments, then:
    $$ \E\left[\sup_{t\in[0,T]} | X_t |^p \right]\leq +\infty.$$ 
\end{Proposition}
\begin{proof}
    The proof can be found in Section \ref{sec:propSDE}.
\end{proof}

\section{Continuity with respect to the kernel}\label{sec:continuity}

We now show that a process following the Hawkes jump-diffusion dynamics \eqref{eq:EDS} is stable with respect to a small perturbation in the kernel $\phi$, provided that the stability assumption \ref{ass:stability} remains in force. 
We start by showing the strongest continuity result in $\mathbb L^2$ if the process $X$ does not impact the intensity $\lambda$, and then we prove the continuity in $\mathbb L^1$ for the general case.
\subsection{State independent intensity}
The independence of the intensity $\lambda$ of the process $X$ makes the system \eqref{eq:EDS} triangular, in the sense that for a given $\Pi$ we can construct a path $\lambda$ independently of $X$, which can then be seen as a classical "jump-diffusion". This means that standard techniques such that Gr\"onwall's inequality can be applied. \\
Given two kernels $\phi$ and $\tilde \phi$ as well as a Brownian motion $W$ and a Poisson measure $\Pi$ we construct two processes $(X,\lambda)$ and $(\tilde X,\tilde \lambda)$ following the dynamics 
  \begin{equation}
  \label{eq:EDS_indep}
    \begin{cases}
        X_t &= x_0 + \int_0^t \mu(s,X_s) \d s + \int_0 ^t \sigma(s,X_s) \d W_s + \int_0^t \int_{\R_+ \times \R} y\gamma(s,X_{s-}) \boldsymbol{1}_{\theta \leq \lambda_s} \Pi (\d s, \d \theta, \d y)\\
        \lambda_t &= \lambda_{\infty}(t)+\psi \left(\int_0^{t-} \int_{\R_+ \times \R}\phi (t-s) \nu(s) b(y) \boldsymbol{1}_{\theta \leq \lambda_s} \Pi (\d s, \d \theta ,\d y) \right)
    \end{cases}
\end{equation}
and 
\begin{equation}
\label{eq:EDS_indep_tilde}
    \begin{cases}
        \tilde X_t &= x_0 + \int_0^t \mu(s,\tilde X_s) \d s + \int_0 ^t \sigma(s,\tilde X_s) \d W_s + \int_0^t \int_{\R_+ \times \R} y\gamma(s,\tilde X_{s-}) \boldsymbol{1}_{\theta \leq \tilde \lambda_s} \Pi (\d s, \d \theta, \d y)\\
        \tilde \lambda_t &= \lambda_{\infty}(t)+\psi \left(\int_0^{t-} \int_{\R_+ \times \R}\tilde\phi (t-s)\nu(s) b(y) \boldsymbol{1}_{\theta \leq \tilde \lambda_s} \Pi (\d s, \d \theta ,\d y) \right)
    \end{cases}.
\end{equation}
We start by proving that the difference between $\tilde \phi$ and $\phi$ in $\mathbb L^1$ and $\mathbb L^2$ control the difference between the intensities $\tilde \lambda$ and $\lambda$. 
\begin{Proposition}
    
\label{prop:continuity_lambda_noX}
  Assume that Assumptions \ref{ass:Lipschitz} and \ref{ass:gronwall} hold. Furthermore, suppose that Assumption \ref{ass:stability} is satisfied  both by $\phi$ and $\tilde \phi$. We then have a constant $C$ that does not depend on $T$ such that  
$$\E \left[\left | \tilde \lambda_t-\lambda_t\right| \right] \leq C\|\tilde \phi-\phi\|_1 \quad \text{and} \quad \E \left[ \left| \tilde \lambda_t-\lambda_t\right|^2\right] \leq C \left(\|\tilde \phi-\phi\|_1+\|\tilde\phi-\phi\|_2^2\right),$$
for all $t \in [0,T].$
\end{Proposition}
\begin{proof}
    The proof can be found in Section \ref{sec:thmcontinuity_lambda_noX}.
\end{proof}
We point out here that the stability Assumption \ref{ass:stability} ensures that the two paths $\lambda$ and $\tilde \lambda$ differ by a constant that does not depend on time in $\mathbb L^1$ and $\mathbb L^2$. The lack of such an assumption would mean the need to resort to a Gr\"onwall type of argument that produces a bound of order $\mathcal{O}(e^{CT})$ for the error observed on a time interval $[0,T]$.\\
Such a scenario would be highly undesirable from a numerical point of view as it would yield a bound of order $\mathcal{O}(\exp(C e^{CT^2}))$ on the difference between $X$ and $\tilde X$.
\begin{Theorem}
\label{thm:continuity_X_noX}
    Assume that Assumptions \ref{ass:Lipschitz} and \ref{ass:gronwall} hold. Furthermore, suppose that Assumption \ref{ass:stability} is satisfied  both by $\phi$ and $\tilde \phi$. 
    Let $(X, \lambda)$ and $(\tilde X, \tilde \lambda)$ be solutions of Equations \eqref{eq:EDS_indep} and \eqref{eq:EDS_indep_tilde} respectively. 
    We then have constants $C_1$ and $C_2$ that do not depend on $T$ such that 
    $$\E \left[\sup_{t\in[0,T]}\left |\tilde X_t-X_t \right|^2\right]\leq C (\|\tilde \phi-\phi\|_1+\|\tilde \phi-\phi\|_2^2)e^{CT^2}$$
\end{Theorem}
\begin{proof}
The proof can be found in Section \ref{sec:thmcontinuity_X_noX}.
\end{proof}

\subsection{State dependent intensity}
Unlike in the previous subsection, the presence of $X$ in the dynamics that drive the intensity $\lambda$ entangles the $\mathbb L^1$ character of the driving Poisson randomness $\Pi$ with the $\mathbb L^2$ character of the Brownian noise. This renders the classic Gr\"onwall techniques useless because of the presence of square roots. We thus need to resort to different contraction arguments to prove continuity. 
\begin{Proposition}
\label{prop:continuity_lambda_depend_X}
    Given two kernels $\phi$ and $\tilde\phi$, assume that Assumptions \ref{ass:Lipschitz}, \ref{ass:stability} and \ref{ass:gronwall} hold. 
    Let $(X,\lambda)$ and $(\tilde X,\tilde \lambda)$ be the solutions of SDE \eqref{eq:EDS} with kernels $\phi$ and $\tilde \phi$ respectively. There exists a non negative function $R\in \mathbb L^1(\R_+)$ and a positive constant $C$ such that 
    $$\E \left[\left | \tilde \lambda_t-\lambda_t\right|\right] \leq C \left(\left\|\tilde \phi-\phi \right \|_1+\int_0^t R(t-s)\left(\E |\lambda_{\infty}(s,\tilde X_s)-\lambda_{\infty}(s,X_s)|+\E |\nu(s,\tilde X_s)-\nu(s,X_s)|\lambda_s\right)\d s\right).$$
    And as a consequence 
    $$ \E \left[\left | \tilde \lambda_t-\lambda_t\right|\right] \leq C \left( \|\tilde \phi-\phi\|_1+\E \left[\sup_{s\in[0,t]} \left |\tilde X_s -X_s \right|\right]\right).$$
    
\end{Proposition}
\begin{proof}
 The proof can be found in Section \ref{sec:propcontinuity_lambda_depend_X}.
\end{proof}

\begin{Theorem}
\label{thm:continuity_X}
    If Assumptions \ref{ass:Lipschitz}, \ref{ass:stability} and \ref{ass:gronwall} hold, then there exist positive constants $C_1$ and $C_2$ that do not depend on $T$ such that 
    $$\E \left[\sup_{s \in [0,T]}\left |\tilde X_s-X_s \right |\right] \leq C _1\| \tilde \phi-\phi\|_1e^{C_2 T}.$$
\end{Theorem}

\begin{proof}
 The proof can be found in Section \ref{sec:thmcontinuity_X}.
\end{proof}

\section{The Markov approximation of Hawkes jump-diffusions}\label{sec:main}

Due to the presence of the convolution in the intensity's definition in Equation \eqref{eq:EDS}, the stochastic process $X$ -as well as $(X,\lambda)$- is not a Markov process. This means that the standard techniques used for stochastic control no longer apply. \\
However, we can show that if the kernel belongs to a certain family of $\mathbb L^1(\R_+) \cap\mathbb L^2(\R_+)$ functions, one can retrieve the Markov property , albeit at the cost of adding auxiliary processes. \\
A function $\phi$ is called a sum of exponentials if there exists a positive integer $n$, a set of coefficients $(\eta_1,\cdots,\eta_n) \in \R^n$ and of exponents $(\beta_1,\cdots,\beta_n) \in \R_+^n$ such that 
\begin{equation}
    \label{eq:sum_of_exponentials}
    \phi(t)=\sum_{k=1}^n \eta_k e^{-\beta_kt}, \quad \text{for all }t \in \R_+.
\end{equation}
Without loss of generality, we assume that $0<\beta_1 <\cdots<\beta_n$.\\
We now take $(X,\lambda)$ to be a solution of \eqref{eq:EDS}, with $\phi$ defined by \eqref{eq:sum_of_exponentials}. In this case, the intensity takes the form
\begin{align*}
    \lambda_t &= \lambda_{\infty}(t,X_{t-})+\psi \left(\int_0^{t-} \int_{\R_+ \times \R}\phi (t-s) b(y) \nu (s,X_{s-})\boldsymbol{1}_{\theta \leq \lambda_s} \Pi (\d s, \d \theta ,\d y) \right)\\
    &= \lambda_{\infty}(t,X_{t-})+\psi \left(\int_0^{t-} \int_{\R_+ \times \R}\sum_{k=1}^n \eta_k e^{-\beta_k(t-s)} b(y) \nu (s,X_{s-})\boldsymbol{1}_{\theta \leq \lambda_s} \Pi (\d s, \d \theta ,\d y) \right)\\
    &=\lambda_{\infty}(t,X_{t-})+\psi \left(\sum_{k=1}^n\eta_k \xi^{k}_t \right),\\
\end{align*}
where the auxiliary processes $(\xi^{j}_t)_{t\in \R_+}$ are defined by 
\begin{equation}
    \label{eq:auxiliary}
    \xi^{j}_t=\int_0^{t-}\int_{\R_+ \times \R} e^{-\beta_j(t-s)}b(y)\nu(s,X_{s-})\boldsymbol{1}_{\theta \leq \lambda_s} \Pi(\d s, \d \theta ,\d y).
\end{equation}
Thanks to the flow property of the exponential function, we have that 
$$ e^{\beta_j t}\xi^{j}_t=\int_0^{t-}\int_{\R_+ \times \R} e^{\beta_j s}b(y)\nu(s,X_{s-})\boldsymbol{1}_{\theta \leq \lambda_s} \Pi(\d s, \d \theta ,\d y),$$
which by differentiation yields 
$$\d \xi^{j}_t+\beta_j\xi^{j}_t \d t=\int_{\R_+ \times \R}b(y)\nu(t,X_{t-})\boldsymbol{1}_{\theta \leq \lambda_t}\Pi(\d t, \d \theta,\d y).$$
Therefore, the augmented process $(X,\xi^{1},\cdots,\xi^{n})$ solves the first order SDE 
\begin{equation}
    \label{eq:SDE_markov}
    \begin{cases}
        \d X_t=&\mu(t,X_t) \d t+\sigma(t,X_t)\d W_t+\int_{\R_+\times \R}y \gamma(t,X_{t-})\boldsymbol{1}_{\theta \leq \lambda_\infty(t,X_{t-})+\psi \left(\sum_{k=1}^n\eta_k \xi^{k}_t \right)} \Pi(\d t,\d \theta,\d y)\\
        \d \xi^{1}_t=&- \beta_1 \xi_{t}^{1} \d t + \int_{\R_+ \times \R} b(y)\nu(t,X_{t-})\boldsymbol{1}_{\theta \leq \lambda_\infty(t,X_{t-})+\psi \left(\sum_{k=1}^n\eta_k\xi_t^{k}\right)} \Pi(\d t,\d \theta,\d y)\\
        \vdots  \\
        \d \xi^{n}_t=&- \beta_n \xi_{t}^{n} \d t + \int_{\R_+ \times \R} b(y)\nu(t,X_{t-})\boldsymbol{1}_{\theta \leq \lambda_\infty(t,X_{t-})+\psi \left(\sum_{k=1}^n\eta_k\xi_t^{k}\right)} \Pi(\d t,\d \theta,\d y)\\
    \end{cases}
    .
\end{equation}
Such a process is therefore a Markov process whose infinitesimal generator $\mathcal A$ defined for any smooth function $g : [0,T] \times\R^{n+1}\to\R$ by
\begin{align}
    \label{eq:generator}
    \mathcal A g(t,x,\xi_1,\cdots,\xi_n):=&\frac{\partial g}{\partial t}+\mu(t,x)\frac{\partial g}{\partial x}-\sum_{k=1}^n \beta_k \xi_k\frac{\partial g}{\partial \xi_k} + \frac{1}{2}\sigma ^2(t,x)\frac{\partial ^2g}{\partial x^2}\nonumber \\
    &+\left(\lambda_\infty(t,x)+\psi \left(\sum_{k=1}^n \eta_k \xi_k\right)\right) \bigg(\int_{\R}g\left(t,x+\gamma(t,x)y,\xi_1+\nu(t,x)b(y),\cdots,\xi_n+\nu(t,x)b(y)\right)\nonumber \\
    &-g(t,x,\xi_1,\cdots,\xi_n) m (\d y )\bigg). \nonumber \\
\end{align}
For such dynamics, it is then possible to write the Hamilton-Jacobi-Bellman for the augmented process $(X,\xi^{1},\cdots, \xi^{n})$ and one can use the classical techniques for stochastic control on jump-diffusion equations. \\
We now show that such processes are universal approximators of jump-diffusion processes driven by Hawkes processes. Similarly to observation 2.9 in \citeauthor{DLO} \cite{DLO}, this is based on the following two facts: 
\begin{itemize}
    \item Sums of exponentials are universal approximators of integrable functions, and not only of completely monotonous ones. 
    \item Hawkes jump-diffusions are continuous with respect to their kernel. 
\end{itemize}
\begin{Theorem}\label{thm:approximation}
    Assume that Assumptions \ref{ass:Lipschitz},\ref{ass:stability} and \ref{ass:gronwall} are in force. Given a time horizon $T>0$, let $(X,\lambda)$ be a solution of the SDE \eqref{eq:EDS} on $[0,T]$.  \\
    For any  $\varepsilon >0$, there exists an $n \in \N$ and a $n+1$-dimensional Markov process $(\tilde X, \tilde \xi^{1},\cdots, \tilde\xi^{n})$ that follows the dynamics \eqref{eq:SDE_markov} such that 
    $$\E \left[\sup_{s\in[0,T]} \left|X_s-\tilde X_s \right|\right]\leq \varepsilon .$$
\end{Theorem}
\begin{proof}
    According to Theorem \ref{thm:continuity_X}, there exist positive $C_1$ and $C_2$ such that for any solution of \eqref{eq:EDS} with the same parameters and a different kernel $\tilde \phi$ satisfying Assumptions \ref{ass:Lipschitz} ,\ref{ass:stability} and \ref{ass:gronwall} we have that 
    $$\E \left[\sup_{s\in[0,T]} \left|X_s-\tilde X_s \right|\right]\leq C_1 \|  \phi- \tilde\phi\|_1e^{C_2T}.$$
    According to the corollary of Theorem 1 in \cite{Kammler} or Theorem 1 in \cite{AK}, for any $ \beta >0$ the set of finite linear combinations of exponentials $\left \{ \left(e^{-\beta k t}\right)_{t \in \R_+}\right \}_{k \geq 1}$ is dense in $\mathbb L^p(\R_+)$ for any $p\geq 1$. Thus, we can find a function $$\tilde \phi(t)=\sum_{k=1}^n \eta_k e^{-\beta k t}$$ such that 
    $$\|\phi-\tilde \phi\|_1\leq \frac{\varepsilon e^{-C_2 T}}{C_1},$$
    hence the result. 
\end{proof}

In practice, the approximation of a given kernel $\phi \in \mathbb L^1(\R_+)$ can be approached in the following way: 
\begin{itemize}
    \item While in theory the linear combinations of $\left \{\left(e^{-\beta k t}\right)_{t \in \R_+}\right\}_{k\ge 1}$ is dense in $\mathbb L^1(\R_+)$ for any fixed choice $\beta > 0$, this quantity be seen as a decay parameter that determines the length of the time span after which the effect of a past event becomes negligible. This means that a more appropriate choice of $\beta$ will lead to a better approximation with fewer exponentials.
    \item Once $\beta>0$ is fixed, we choose the number $n \ge 1$ of approximating exponentials. Naturally a higher number of terms has the potential to better approximate the kernel $\phi$, but comes at the price of increased complexity. For instance, the cost of simulation of the dynamics \eqref{eq:SDE_markov} is of order $\mathcal{O}(nT)$. This is still cheaper than the $\mathcal{O}(T^2)$ needed to simulate the original SDE \eqref{eq:EDS}.
    \item We then optimise the convex function 
    $$J(\boldsymbol{\eta}):=\int_0^{+\infty}\left | \sum_{k=1}^n \eta_k e^{-\beta k t}-\phi(t)\right | \d t .$$
    This can be done either:
    \begin{enumerate}
        \item Directly numerically, for instance using the method \texttt{scipy.optimize.minimize} coupled with the preferred quadrature method in Python.
        \item By optimising the function
        $$I(\boldsymbol{\eta}):=\int_0^{+\infty}\left ( \sum_{k=1}^n \eta_k e^{-\beta k t}-\phi(t)\right )^2 \d t $$
        which has advantage of being differentiable. In fact, the equation $\nabla_{\boldsymbol{\eta}}I(\boldsymbol{\eta})= \boldsymbol{0}$ has a unique solution $\boldsymbol{\eta}^*$ which can be determined by inverting the system 
        $$M_H \boldsymbol{\eta}^*=\boldsymbol{v}$$
        where $\boldsymbol{v}=\left(\int_0^{+\infty}e^{-\beta j t}\phi(t) \d t\right)_{j=1,\cdots,n}$ and $M_H=\left(\frac{1}{\beta(i+j)}\right)_{i,j=1,\cdots,n}$ is a modified Hilbert matrix. \\
        We then use the fact that $J(\boldsymbol{\eta}) \le I(\boldsymbol{\eta})^{1/2}$ (an immediate application of the Cauchy-Schwarz inequality) to obtain an approximation in $\mathbb L^1(\R_+)$ and $\mathbb L^2(\R_+)$. 
    \end{enumerate}
    \item If the desired precision is not obtained, $n$ should be increased and a new set of coefficients $\boldsymbol{\eta}$  should be computed. We point out that the exponential functions are not orthogonal in $\mathbb L^2(\R_+)$.
    \item We do not optimise over $\beta>0$ because of the lack of convexity. This means that if we optimise on $\beta$, the numerical methods are no longer guaranteed to converge to the global minimiser. 
\end{itemize}

\begin{Remark}
    It is also possible to prove that for any $\beta >0$, the span of $\left\{x^ke^{-\beta x}\right\}_{k\ge 0}$ is also dense in $\mathbb L^1(\R_+) \cap \mathbb L^2(\R_+)$, \textit{cf.} \cite{AK} for instance. This family also has the Markov property, due to the fact that $f_{k}:x \mapsto x^{k}e^{-\beta x}$ satisfies the first order ODE
    $$\frac{\d }{\d x}f_k (x)=k f_{k-1}(x)-\beta f_k(x).$$
\end{Remark}

The curious reader can also adapt the approach introduced by Beylkin and Monz\'on in \cite{Beylkin} to exponential sums with real coefficients and exponents. This method might have the advantage of providing a more accurate approximation as it also optimises on the decay parameters $(\beta_k)_{k \ge 1}$.\\

To illustrate the Markov approximation discussed above, we consider the following modified Hawkes-Ornstein-Uhlenbeck SDE
\begin{equation}
    \label{eq:Hawkes_OE}
    \begin{cases}
        \d X_t & = -\mu X_t \d t+\sigma \d W_t + \gamma(X_{t-})\int_{\R_+}\boldsymbol{ 1} _ {\theta \le \lambda_t} \Pi (\d t, \d \theta) \\
        \lambda_t &= \lambda_\infty + \psi \left(\int_0^{t-}\int_{\R_+}\phi(t-s) \nu(X_s)\boldsymbol{1}_{\theta \leq \lambda_s} \Pi(\d s , \d \theta)\right)
    \end{cases}
\end{equation}
where $\mu=0.5$, $\sigma=1$, $\gamma(x)=-(40x)/(1+16x^2)$ ,$\lambda_\infty=1$, $\psi(x)=(x)_+\wedge 7$ and $\nu(x)=0.2+0.8\exp(-0.1 x^2)$. 
The kernel is chosen to be 
$$\phi(t)=\frac{1-t}{(1+t^{2.5})},$$
which is not monotonous (\textit{a fortiori} not completely monotonous) and exhibits both excitation and inhibition as can be seen in Figure \ref{fig:kernel}. This kernel is then approximated by a linear combination of $2$ and $3$ exponentials. First, we choose the decay parameter $\beta =0.5$ which gurantees that the exponential functions become negligible for $t\ge 10$ (we assume that any quantity below $e^{-5}$ is negligible).  We then approximate $\phi$ using the methodology explained above, yielding 
$$\phi^{(2)}=-1.16 e^{-0.5t}+2.17e^{-t}$$
and 
$$\phi^{(3)}(t)=-0.82 e^{-0.5 t}+0.58e^{-t}+1.39e^{1.5t} .$$

\begin{figure}[h!]
    \centering
    \includegraphics[width=0.8\linewidth]{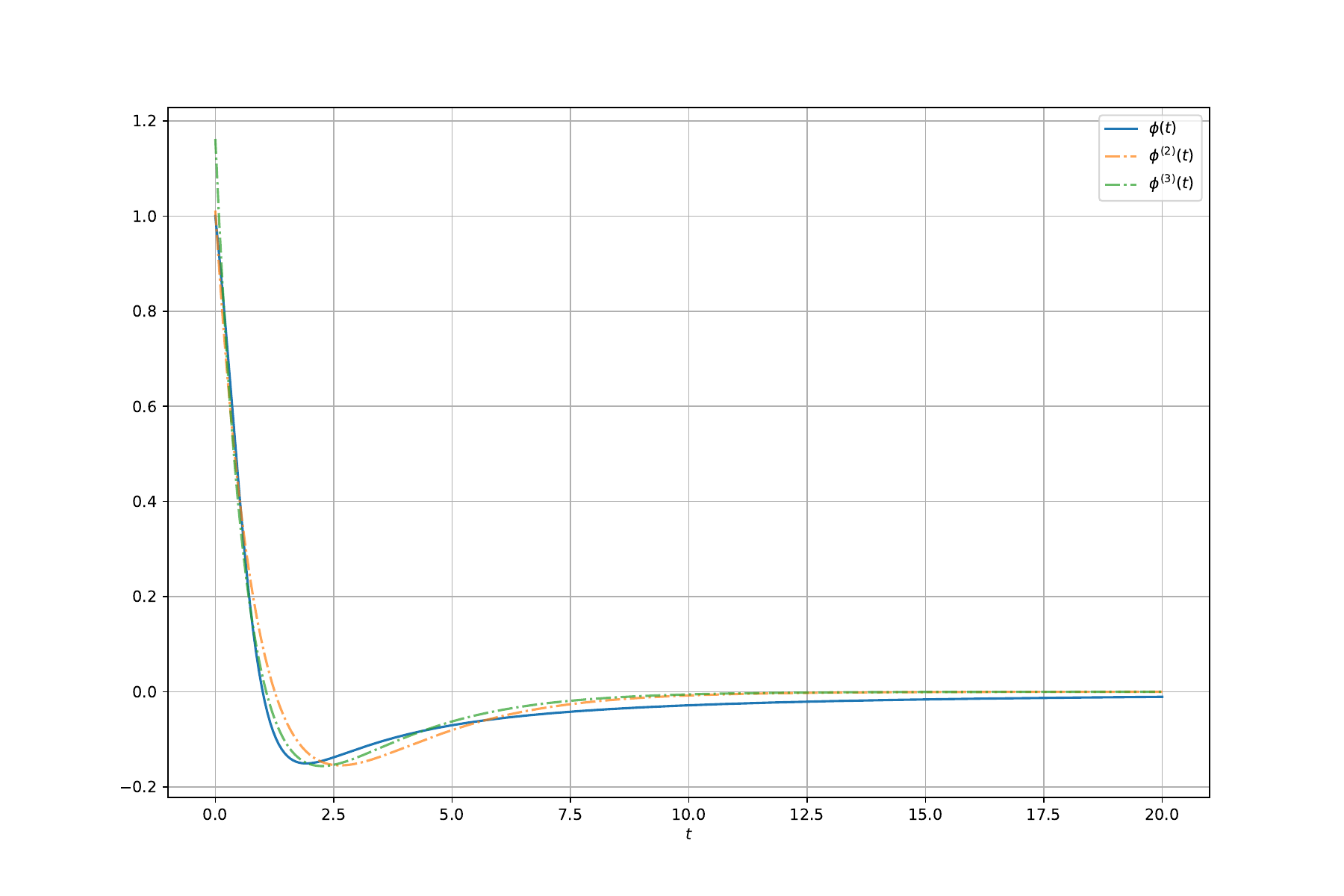}
    \caption{$\phi^{(3)}$ does a better job at approximating $\phi$ than $\phi^{(2)}$. Both approximations fail at capturing the heavy polynomial tail, but this is not a problem given that we are interested in approximation on finite horizons.}
    \label{fig:kernel}
\end{figure}

We point out that the first two coefficients of $\phi^{(2)}$ are not preserved in $\phi^{(3)}$, due to the lack of orthogonality. \\
Given a Brownian motion $W$ and a Poisson measure $\Pi$ we simulate three paths $(X,\lambda)$, $(X^{(2)},\lambda^{(2)})$ and $(X^{(3)},\lambda ^{(3)})$ using the kernels $\phi$, $\phi^{(2)}$ and $\phi^{(3)}$ in the SDE \eqref{eq:Hawkes_OE} respectively. The paths are shown in Figure \ref{fig:xl}. 
\begin{figure}[h!]
    \centering
    \includegraphics[width=1\linewidth]{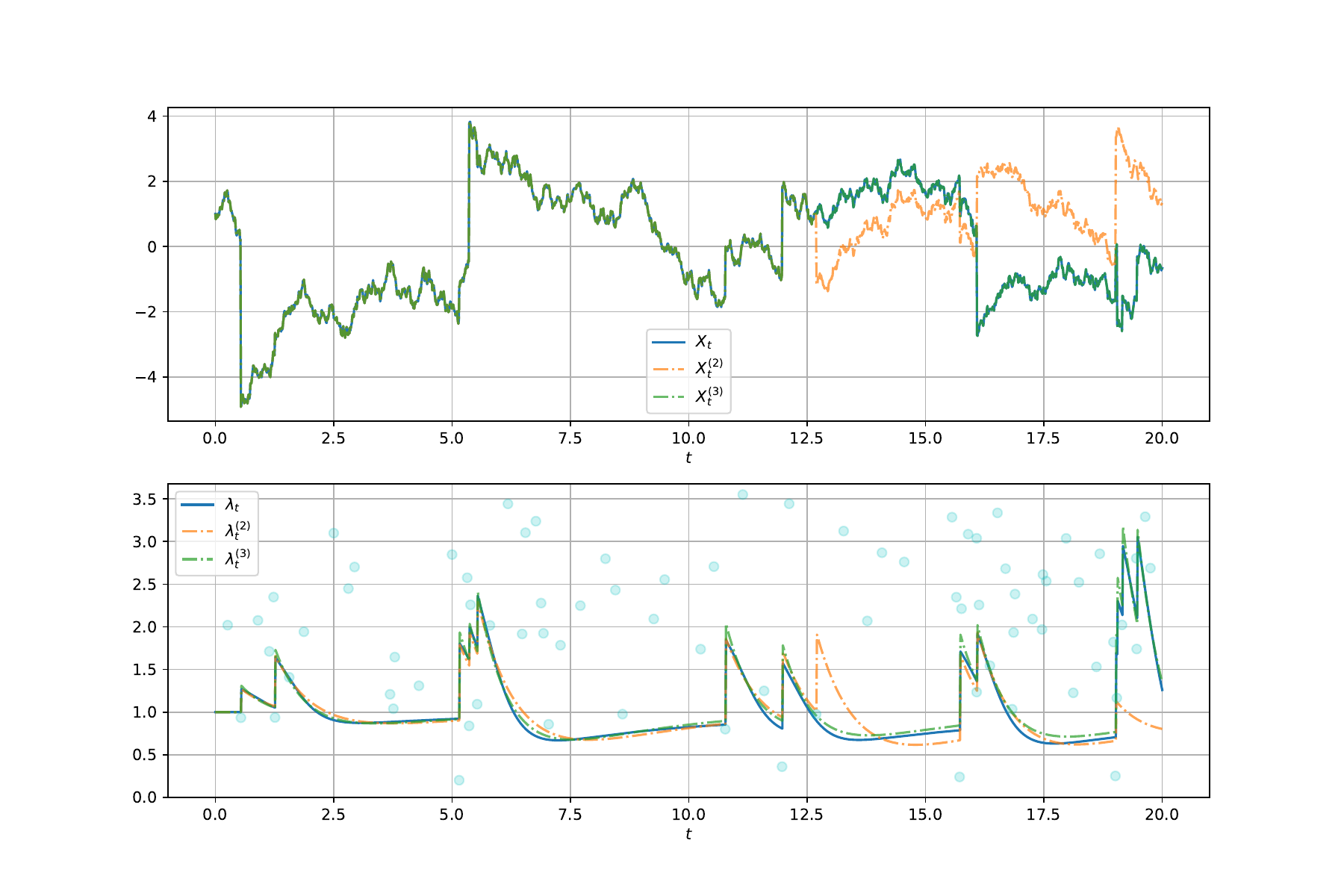}
    \caption{A realisation of $(X,\lambda)$ on $[0,T]$. The order 3 Markov approximation $X^{(3)}$ (in green) remains confounded with the original process $X$ (in blue). This is because $\lambda^{(3)}$ is close enough to $\lambda $ to accept and reject the exact same point of $\Pi$ (in cyan). The same cannot be said about the order 2 Markov approximation (in orange).}
    \label{fig:xl}
\end{figure}
\section{Application to stochastic control of Hawkes jump-diffusions}\label{sec:control}

\subsection{The control problem}

In this paragraph, we denote by $\mathbb{F} := \{\mathcal{F}_t\}_{0 \le t \le T}$ the filtration generated by $(W, \Pi)$.
Let $A$ be a Polish space. 
Given $\alpha : [0,T] \times \Omega \longrightarrow A$, we then consider the controlled version of the dynamics \eqref{eq:EDS}:
\begin{equation}
\label{eq:controlled-EDS}
    \begin{cases}
        X_t &= x_0 + \int_0^t \mu(s,X_s,\alpha_s) \d s + \int_0 ^t \sigma(s,X_s,\alpha_s) \d W_s + \int_0^t \int_{\R_+ \times \R} y\gamma(s,X_{s-},\alpha_{s-}) \boldsymbol{1}_{\theta \leq \lambda_s} \Pi (\d s, \d \theta, \d y)\\
        \lambda_t &= \lambda_{\infty}(t,X_{t-})+\psi \left(\int_0^{t-} \int_{\R_+ \times \R}\phi (t-s) b(y) \nu (s,X_{s-},\alpha_{s-})\boldsymbol{1}_{\theta \leq \lambda_s} \Pi (\d s, \d \theta ,\d y) \right),
    \end{cases}
\end{equation}
where we omit to write the dependence of $(X, \lambda)$ in $\alpha$ for notational simplicity. We denote by $\cA$ the set of $\mathbb{F}$-adapted processes taking their values in $A$ such that, for all $\varphi \in \{\mu, \sigma, \gamma, \nu\}$ and all control $\alpha \in \cA$, the mapping $(t, x, \omega) \mapsto \varphi(t, x, \alpha_t(\omega))$ satisfies Assumptions \ref{ass:Lipschitz}-\ref{ass:gronwall}. In particular, for all $\alpha \in \cA$, the equation \eqref{eq:controlled-EDS} admits a unique solution, $\mathbb{P}$-a.s., and we have:
\begin{equation}\label{eq:indep-of-alpha}
\E\left[ \sup_{t \in [0,T]} | X_T |^p \right] \le C_p \quad \mbox{for all} \ p \ge 1, \quad \mbox{with} \ C_p \ \mbox{independent of} \ \alpha.
\end{equation}
Now, given $f : [0,T] \times \mathbb{R}^d \times A \to \mathbb{R}$ and $g : \mathbb{R}^d \to \mathbb{R}$, we define the optimal control problem:
\begin{equation}\label{eq:control-pb}
V_0 := \sup_{\alpha \in \mathcal{A}} \mathbb{E}\Big[\int_0^T f(t, X_t, \alpha_t)\d t + g(X_T) \Big]
\end{equation}

As said in the Introduction, due to the possibly nonexponential kernel in the dynamics of the intensity $\lambda$, \eqref{eq:control-pb} is a Volterra-type control problem, and therefore cannot be solved through the standard dynamic programming techniques. This motivates us to introduce a Markov approximation of $(X,\lambda)$.

\subsection{The Markov approximation}

Let $\phi^n(t) := \sum_{k=1}^n \eta_k e^{-\beta_k t}$ be such that $\lVert \phi^n - \phi \rVert_1 \to 0$ as $n \to +\infty$. We denote by $(X^n, \lambda^n)$ the controlled Hawkes jump-diffusion \eqref{eq:controlled-EDS} driven by the kernel $\phi^n$ instead of $\phi$. We consider the control problem:
\begin{equation}\label{eq:markov-control-pb}
V_0^n := \sup_{\alpha \in \mathcal{A}} \mathbb{E}\left[\int_0^T f(t, X_t^n, \alpha_t)\d t + g(X_T^n) \right]. 
\end{equation}
As seen in Section \ref{sec:main}, introducing the processes $(\xi^1, \dots, \xi_n)$ defined in \eqref{eq:auxiliary}, we have $$\lambda_t^n = \lambda_\infty(t, X_{t-}) + \psi\left( \sum_{k=1}^n \eta_k \xi_t^k \right),$$ and the dynamics $(X^n, \xi^1, \dots, \xi^n)$ is actually a Markov jump-diffusion, which allows us to analyze the control problem \eqref{eq:markov-control-pb} using the dynamic programming approach. In particular, under mild regularity assumptions, it is standard to verify that $V_0^n = v(0, X_0, 0, \dots, 0)$, where $v : [0,T] \times \R^{n+1} \longrightarrow \R$ is a viscosity solution to the Hamilton-Jacobi-Bellman equation:
\begin{align*}
    - \partial_t v(t, \boldsymbol{\xi}) - \mathcal{H}[v](t, x, \boldsymbol{\xi}) = 0, \quad v|_{t=T} = g,
\end{align*}
where:
\begin{align*}
   \mathcal{H}[v](t, x, \boldsymbol{\xi}) := \sup_{a \in A} \Big\{ &\mu(t, x, a)\partial_x v(t,x,\boldsymbol{\xi}) + \frac12 \sigma^2 (t,x,a) \partial_{xx}^2 v(t,x,\boldsymbol{\xi}) - (\boldsymbol{\beta \xi}) \cdot \nabla_{\boldsymbol{\xi}} v(t,x,\boldsymbol{\xi}) \\&+ \big(\lambda_\infty(t,x) + \psi(\boldsymbol{\eta} \cdot \boldsymbol{\xi})\big)\int_{\R} \big( v(t, x+\gamma(t,x,a)y, \boldsymbol{\xi} + \nu(t,x,a)b(y)\boldsymbol{1}) - v(t,x,\boldsymbol{\xi})\big)m(\d y) \Big\},
\end{align*}
with the notations $\boldsymbol{y} := (y_1, \dots, y_n) \in \R^n$ and $\boldsymbol{1} = (1, \dots, 1) \in \R^n$. We refer to \citeauthor{oksendal2007applied} \cite{oksendal2007applied} for more information on the control of Markov jump-diffusion processes. \\
We now state the main result of this paragraph:
\begin{Theorem}\label{thm:convergence-value-function}
Assume $f$ and $g$ are continuous with polynomial growth in $x$ of order $m \ge 1$, uniformly in $(t,a)$. Then we have:
$$ V_0^n \underset{n \to \infty}{\longrightarrow} V_0. $$
In particular, if $f$ is Lipschitz-continuous in $x$, uniformly in the other variables, and $g$ is Lipschitz-continuous, we have:
$$ |V_0^n - V_0 | \le \lVert \phi^n - \phi \rVert_1 C_1 e^{C_2 T} $$
for some nonnegative constants $C_1, C_2$. 
\end{Theorem}
\proof 
The proof can be found in Section \ref{sec:thmcvgvalue}.
\qed 

\subsection{Example: portfolio optimization in markets with contagion}

We revisit the study of \citeauthor{ait2015portfolio} \cite{ait2015portfolio}, who analyzed the problem of optimization a portfolio driven by a Hawkes jump-diffusion process. In their framework, the intensity of the Hawkes process is driven by an exponential kernel, which induces a Markov dynamics. We presently discuss the case of a general nonexponential kernel. 

\paragraph*{General utility.} Consider a risky asset $S$ driven by the following dynamics:
$$ \frac{\d S_t}{S_{t-}} = \mu \d t + \sigma \d W_t + \gamma \d N_t,
$$
where $\d N_t = \int_{\R_+} \boldsymbol{1}_{\{ \theta \le \lambda_t \}}  \Pi(\d t, \d\theta)$ and the intensity $\lambda$ satisfies:
\begin{equation}\label{eq:intensity}
 \lambda_t = \lambda_0 + \int_0^{t-} \phi(t-s)\d N_s, 
 \end{equation}
where $W$ is a standard Brownian motion, $\Pi$ an independent Poisson measure and $\phi \in \mathbb{L}^1(\R_+)$. An investor holds a proportion $\omega_t \in [0,1]$ of his capital in the risky asset $S$ at time $t$. The remaining proportion $1-\omega_t$ is invested in a risk-free asset, which evolves according to a constant interest rate $r$. At time $t$, the investor also consumes an amount $c_tX_t \d t$ of his capital. His wealth is then represented by the dynamics:
\begin{equation}\label{eq:wealth}
    \frac{\d X_t}{X_{t-}} = (r + (\mu-r) \omega_t  -c_t) \d t + \sigma \omega_t  \d W_t + \gamma \omega_t \d N_t, \quad X_0 > 0. 
\end{equation}
We denote by $\cA$ the set of $\mathbb{F}$-predictable processes $(c, \omega)$ such that $c_t > 0$ and $\omega_t \in [0,1]$ a.s. for all $t \in [0,T]$. The investor wants to choose his consumption $c$ and his investment strategy $\omega$ in order to solve:
\begin{equation}\label{eq:investor-pb}
    V_0 := \sup_{(c, \omega) \in \cA} \E\Big[\int_0^\infty e^{-\rho t} U(c_t X_t)\d t\Big],
\end{equation}
with $\rho > 0$ and $U : \R_+ \to \R$ some utility function. Since $\phi$ is a general nonexponential kernel, the above control problem is path-dependent, and therefore dynamic programming cannot be applied directly. According to Theorem \ref{thm:approximation}, we then introduce a sequence of kernels $\{\phi^n\}_{n \ge 0}$ such that:
$$ \lVert \phi^n - \phi \rVert_1 \underset{n \to \infty}{\longrightarrow} 0 \quad \mbox{and} \quad \phi^n(t) = \sum_{k=1}^n \eta_k e^{-\beta_k t}. $$
Denote by $X^n$ the solution of \eqref{eq:wealth} corresponding to the kernel $\phi^n$, and introduce the processes
$$ \xi_t^i = \int_0^{t-} \int_{\R_+} e^{-\beta_i (t-s)} \boldsymbol{1}_{\{\theta \le \lambda_s^n\}} \Pi(\d s, \d\theta), \quad 1 \le i \le n,  $$
with $\lambda_t^n $ solution to \eqref{eq:intensity} with kernel $\phi^n$. Then, the system $(X^n, \xi^1, \dots, \xi^n)$ satisfies the Markovian jump-diffusion dynamics:
\begin{equation}\label{eq:approximated-dynamics}
\left\{\begin{array}{ll}
\frac{\d X_t^n}{X_{t-}^n} = (r + (\mu-r) \omega_t  -c_t) \d t + \sigma \omega_t  \d W_t + \gamma \omega_t \d N_t^n, \quad X_0^n = X_0, \\
\d\xi_t^i =  - \beta_i \xi_t^i \d t + \d N_t^n, \quad \xi_0^i = 0, \quad 1 \le i \le n,
\end{array}\right.
\end{equation}
with $N^n$ a Poisson process of intensity $\lambda_t^n = \lambda_0 + \sum_{k=1}^n \eta_k \xi_t^k$. The corresponding investor's problem writes:
\begin{equation}\label{eq:investor-pb-approximated}
    V_0^n := \sup_{(c, \omega) \in \cA} \E\Big[\int_0^\infty e^{-\rho t} U(c_t X_t^n)\d t\Big].
\end{equation}
Since $(X^n, \xi^1, \dots, \xi^n)$ is Markov, we formally have $V_0 = v(X_0^n, 0, \dots, 0)$, with $v : \R^{n+1} \to \R$ solution to the dynamic programming equation:
\begin{equation}\label{eq:investor-DPE}
     - \mathcal{H}[v](x,\boldsymbol{\xi}) = 0,
\end{equation}
where:
\begin{align*}
\mathcal{H}[v](x,\boldsymbol{\xi}) := \sup_{c \ge 0, \omega \in [0,1]} \Big\{& (r + (\mu-r) \omega  - c)x \partial_x v(x,\boldsymbol{\xi})  - (\boldsymbol{\beta} \boldsymbol{\xi}) \cdot \nabla_\xi v(x,\boldsymbol{\xi}) + \frac12 \sigma^2 \omega^2 x^2 \partial_{xx}^2 v(x,\boldsymbol{\xi})  - \rho v(x,\boldsymbol{\xi}) \\&+ U(cx) + (\lambda_0 + \boldsymbol{\eta} \cdot \boldsymbol{\xi})\big[v(x + \gamma \omega x, \boldsymbol{\xi} + \boldsymbol{1}) - v(x, \boldsymbol{\xi}) \big] \Big\}.
\end{align*}

\paragraph*{Log utility.} We now set $U := \log$. We first observe that $V_0 < \infty$. Indeed, using the fact that
$$ X_t = X_0 \exp\Big(\int_0^t [r + (\mu-r)\omega_s - \frac12 \sigma^2 (\omega_s)^2 - c_s)]\d s + \sigma \omega_s \d W_s + \gamma \omega_s \d N_s  \Big), \quad \mbox{$\mathbb{P}$-a.s.,} $$
we see that:
\begin{align*}
    V_0 \le& \sup_{c \ge 0} \E\Big[\int_0^\infty e^{-\rho t} \big( \log(c_t) - \int_0^t c_s \d s \big)\d t \Big] + \sup_{\omega \in [0,1]} \E\Big[\int_0^\infty e^{-\rho t} \Big( \int_0^t \big( r + (\mu-r)\omega_s - \frac12 \sigma^2 (\omega_s)^2 \big) \d s + \gamma \omega_s \d N_s\Big)\d t \Big] \\
    \le& \sup_{c \ge 0} \E\Big[\int_0^\infty e^{-\rho t} c_t e^{- \int_0^t c_s \d s} \d t \Big] + \int_0^\infty e^{-\rho t}\left( C_1 t + C_2 \E\left[\int_0^t \lambda _s \d s \right] \right)\d t \\
    \le& \sup_{c \ge 0} \E\Big[1 - \int_0^\infty \rho e^{-\rho t} e^{- \int_0^t c_s \d s} \d t \Big] + \int_0^\infty e^{-\rho t} C_3 t\d t < \infty,
\end{align*}
for some constants $\{C_i\}_{1 \le i \le 3}$, where we used the integration by parts formula and Lemma \ref{lmm:moments_lambda}.
As in \cite{ait2015portfolio}, we look for a solution to \eqref{eq:investor-DPE} under the form:
$$ v(x, \boldsymbol{\xi}) = f(\boldsymbol{\xi}) + K \log(x),$$
where $f : \R^n \to \R$ and $K \in \R$ have to be determined. Compute:
\begin{align*}
    \partial_x v(x,\boldsymbol{\xi}) = \frac{K}{x}, \quad \partial_{xx}^2 v(x,\boldsymbol{\xi})) =-\frac{K}{x^2}, \quad \nabla_{\boldsymbol{\xi}} v(x, \boldsymbol{\xi}) = \nabla_{\boldsymbol{\xi}}  f(\boldsymbol{\xi}), \\
    v(x + \gamma \omega x, \boldsymbol{\xi} + \boldsymbol{1}) - v(x, \boldsymbol{\xi}) = K \log(1 + \gamma \omega) + f(\boldsymbol{\xi} + \boldsymbol{1}) - f(\boldsymbol{\xi}),
\end{align*}
and plug it in \eqref{eq:investor-DPE}:
\begin{align*}
      &\sup_{\omega \in [0,1]} \Big\{  (\mu-r) \omega x \partial_x v(x,\boldsymbol{\xi})  + \frac12 \sigma^2 \omega^2 x^2 \partial_{xx}^2 v(x,\boldsymbol{\xi}) + (\lambda_0 + \boldsymbol{\eta} \cdot \boldsymbol{\xi})\big[v(x + \gamma \omega x, \boldsymbol{\xi} + \boldsymbol{1}) - v(x, \boldsymbol{\xi}) \big] \Big\} \\
      &+ \sup_{c \ge 0} \Big\{ -cx\partial_x v(x, \boldsymbol{\xi}) + U(cx) \Big\} + rx\partial_x v(x, \boldsymbol{\xi}) - \rho v(x, \boldsymbol{\xi}) -(\boldsymbol{\beta} \boldsymbol{\xi}) \cdot \nabla_\xi v(x,\boldsymbol{\xi}) \\
      &= \psi(\boldsymbol{\eta} \cdot \boldsymbol{\xi}) + (\lambda_0 + \boldsymbol{\eta}\cdot \boldsymbol{\xi})(f(\boldsymbol{\xi} + \boldsymbol{1}) - f(\boldsymbol{\xi})) + \sup_{c \ge 0} \Big\{ -cK + U(cx) \Big\} + rK - \rho v(x, \boldsymbol{\xi}) -(\boldsymbol{\beta} \boldsymbol{\xi}) \cdot \nabla_\xi f(\boldsymbol{\xi}) = 0,
\end{align*}
with
\begin{align*}
    \psi(\boldsymbol{\eta} \cdot \boldsymbol{\xi}) :=  K\sup_{\omega \in [0,1]} \Big\{ (\mu-r) \omega - \frac12 \sigma^2 \omega^2 + (\lambda_0 + \boldsymbol{\eta} \cdot \boldsymbol{\xi})\log(1 + \gamma \omega) \Big\}.
\end{align*}
Note that he maximum in $c$ is reached in $c = 1/K$, and that all the terms in $x$ cancel in \eqref{eq:investor-DPE} if and only if $K = \frac{1}{\rho}$. Then we obtain:
$$ - \rho f(\boldsymbol{\xi}) -\boldsymbol{\beta \xi} \cdot \nabla_{\boldsymbol{\xi}} f(\boldsymbol{\xi}) + (\lambda_0 + \boldsymbol{\eta}\cdot \boldsymbol{\xi})(f(\boldsymbol{\xi} + \boldsymbol{1}) - f(\boldsymbol{\xi})) = 1 - \frac{r}{\rho} -\log(\rho) - \psi(\boldsymbol{\eta} \cdot \boldsymbol{\xi}). $$
Observe now that the left hand side of the above equation corresponds to the infinitesimal generator of the $n$-dimensional process $\boldsymbol{\xi}_t = (\xi_t^1, \dots, \xi_t^n)$ defined by the dynamics \eqref{eq:approximated-dynamics}. Therefore, by Feynman-Kac formula, we have the probabilistic representation:
$$ f(\boldsymbol{\xi}) = \frac{1}{\rho}\Big(1 - \frac{r}{\rho} - \log(\rho)\Big) -\int_0^\infty e^{-\rho t} \E\Big[ \psi(\boldsymbol{\eta} \cdot \boldsymbol{\xi}_t) | \boldsymbol{\xi}_0 = \boldsymbol{\xi} \Big]\d t, $$
and therefore:
\begin{align*}
V_0^n =& \frac{1}{\rho}\Big( \log(X_0^n) + 1 - \frac{r}{\rho} - \log(\rho)\Big) -\int_0^\infty e^{-\rho t} \E\Big[ \psi(\boldsymbol{\eta} \cdot\boldsymbol{\xi}_t) \Big]\d t  \\
=& \frac{1}{\rho}\Big( \log(X_0) + 1 - \frac{r}{\rho} - \log(\rho)\Big) -\int_0^\infty e^{-\rho t} \E\Big[ \psi(\lambda_t^n - \lambda_0) \Big]\d t
\end{align*}
By Theorem \ref{thm:approximation}, we have:
$$ V_0^n \underset{n \to \infty}{\longrightarrow} \frac{1}{\rho}\Big( \log(X_0) + 1 - \frac{r}{\rho} - \log(\rho)\Big) -\int_0^\infty e^{-\rho t} \E\Big[ \psi(\lambda_t - \lambda_0) \Big]\d t $$
The last result to prove is that $V_0^n$ indeed converges to $V_0$ as $n \to \infty$. 

\paragraph*{Convergence of the value function.} Theorem \ref{thm:convergence-value-function} can not be applied directly as the control problems \eqref{eq:investor-pb} and \eqref{eq:investor-pb-approximated} are formulated in infinite horizon. However, we have, for all $T > 0$:
\begin{align*}
V_0 \le&  \sup_{(c, \omega) \in \cA} \E\Big[\int_0^T e^{-\rho t} U(c_t X_t)\d t\Big]  + e^{-\rho T} \sup_{(c, \omega) \in \cA} \E \Big[\int_T^\infty e^{-\rho(t-T)} U(c_t X_t)\d t\Big],
\end{align*}
We see that both suprema are finite by the same arguments implying the finiteness of $V_0$. Therefore, all $\varepsilon > 0$, we may find $T_\varepsilon > 0$ such that 
\begin{equation}\label{eq:ineq-example-1}
 V_0 \le V_0^\varepsilon + \varepsilon,
 \end{equation}
where
\begin{equation}\label{eq:finite-horizon}
   V_0^\varepsilon := \sup_{(c, \omega) \in \cA} \E\Big[\int_0^{T_\varepsilon} e^{-\rho t} U(c_t X_t)\d t\Big].
\end{equation}
As for $V^n$, we know there exist an optimal control $(c^n, \omega^n)$ such that $c^n_t = \rho$. Therefore, denoting $\hat X$ the optimally controlled dynamics:
\begin{align*}
    V_0^n \le&  \E\Big[\int_0^T e^{-\rho t} U(\rho \hat X_t^n) \d t\Big]  + e^{-\rho T} \E\Big[\int_T^\infty e^{-\rho (t-T)} U(\rho \hat X_t^n)\d t\Big] \\
    \le& \sup_{(c, \omega) \in \cA} \E\Big[\int_0^T e^{-\rho t} U(c_t X_t^n)\d t\Big] + e^{-\rho T} \E\Big[\int_T^\infty e^{-\rho (t-T)} U(\rho \hat X_t^n)\d t\Big]
\end{align*}
By Theorem \ref{thm:approximation}, for any control $(c, \omega)$, $X^n$ converges to $X$ in $\mathbb{L}^1$. In particular, this implies that the family $\{ \hat X_t^n \}_{n \ge 0}$ is uniformly integrable for all $t \ge 0$. Since $U$ has linear growth, we may conclude that there exists $T_\varepsilon$ (independent of $n$) such that:
\begin{equation}\label{eq:ineq-example-2}
    V_0^n \le V_0^{n,\varepsilon} + \varepsilon,
\end{equation}
where 
\begin{equation}\label{eq:finite-horizon2}
   V_0^{n,\varepsilon} := \sup_{(c, \omega) \in \cA} \E\left[\int_0^{T_\varepsilon} e^{-\rho t} U(c_t X_t^n)\d t\right].
\end{equation}
Now, combining \eqref{eq:ineq-example-1} and \eqref{eq:ineq-example-2} with the natural inequalities $V_0^\varepsilon \le V_0$ and $V_0^{n,\varepsilon} \le V_0^n$, we obtain:
\begin{align*}
|V_0^n - V_0 | \le |V_0^{\varepsilon,n} - V_0^\varepsilon | + |V_0^{\varepsilon} - V_0 | + |V_0^{\varepsilon, n} - V_0^n | \le |V_0^{\varepsilon,n} - V_0^\varepsilon |  + \frac{2\varepsilon}{3}.
\end{align*}
Now, since the control problems \eqref{eq:finite-horizon} and \eqref{eq:finite-horizon2} have a finite horizon, by Theorem \ref{thm:convergence-value-function} we have $|V_0^{\varepsilon,n} - V_0^\varepsilon | \to 0$ as $n \to \infty$. Then, for $n$ sufficiently large, we have $|V_0^n - V_0 | \le \varepsilon$.

\paragraph*{Conclusion.} We have $V_0^n \to V_0$ as $n \to \infty$, and therefore:
$$ V_0 = \frac{1}{\rho}\Big( \log(X_0) + 1 - \frac{r}{\rho} - \log(\rho)\Big) -\int_0^\infty e^{-\rho t} \E\Big[ \psi(\lambda_t - \lambda_0) \Big]\d t, $$
with optimal consumption policy $C^*_t = \rho$ and optimal investment policy satisfying:
$$ \omega_t^* \in \frac{1}{\rho}\underset{\omega \in [0,1]}{\mathrm{argmax}} \Big\{ (\mu-r) \omega - \frac12 \sigma^2 \omega^2 + (\lambda_t)\log(1 + \gamma \omega) \Big\}. $$

\subsection{Numerical approach}

We briefly discuss an algorithmic approach to approximate the value function $V_0$ of the optimal control problem \eqref{eq:control-pb}. Our methods consists in the following steps:
\begin{enumerate}
    \item We first approximate $V_0$ by $V_0^n$, the value function of the Markov control problem \eqref{eq:markov-control-pb}.
    \item We formulate a discrete time version of the control problem \eqref{eq:markov-control-pb}.
    \item We compute the value function of the discrete time problem through a backwards algorithm based on the dynamic programming principle.
\end{enumerate}
For the Step 3, it is important to observe that the possibly high dimension of the state process of the Markov problem represents a computational challenge. We may then combine the dynamic programming approach with neural network approximations, similarly to the "HybridNow" algorithm of \citeauthor{hure2021deep} \cite{hure2021deep}. 

\section{Proofs of the main results}\label{sec:proofs}

\subsection{Proof of Proposition \ref{prop:SDE}}
\label{sec:propSDE}
The proof is an adaptation of a classical contraction argument to the $\mathbb L^1$ setting, following the proof of Theorem 1.2 in \cite{graham}. \\
    Let $ \mathcal E$ be the space of paths $(X,\lambda)$ on $[0,T]$ such as $X$ (respectively $\lambda$) is adapted (respectively predictable) to the filtration $\mathbb F$ such that 
    $$\E \left[\sup_{t\in[0,T]} |X_t|\right]+\E \left[\sup_{t\in[0,T]} |\lambda_t|\right] <+\infty.$$
    Let $f$ be the function that maps a process $(X,\lambda) \in \mathcal E$ to the  process $(Z,\kappa) \in \mathcal E$ defined by 
    \begin{equation*}
    \begin{cases}
        Z_t &= x_0 + \int_0^t \mu(s,X_s) \d s + \int_0 ^t \sigma(s,X_s) \d W_s + \int_0^t \int_{\R_+ \times \R} y\gamma(s,X_{s-}) \boldsymbol{1}_{\theta \leq \lambda_s} \Pi (\d s, \d \theta, \d y)\\
        \kappa_t &=  \lambda_{\infty}(t,X_{t-})+ \psi \left(\int_0^{ t-} \int_{\R_+ \times \R}\phi (t-s) b(y) \nu (s,X_{s-})\boldsymbol{1}_{\theta \leq \lambda_s} \Pi (\d s, \d \theta ,\d y) \right)
    \end{cases}.
\end{equation*}
Given two processes $(X,\lambda)$ and $(X',\lambda')$ and their images $(Z,\kappa)$ and $(Z',\kappa')$ we have that 
\begin{align*}
    \E \left[\sup_{t\in [0,T]}|\kappa'_t-\kappa_t|\right] \leq &\E \left [\sup_{t\in[0,T]} \left|\lambda_{\infty}(t,X'_t)-\lambda_\infty(t,X_t) \right|\right] \\
    &+ L\E \left[\sup_{t\in [0,T]} \left| \int_0^{t-} \int_{\R_+ \times \R} \phi(t-s)b(y)\left(\nu (s,X'_{s-})\boldsymbol{1}_{\theta \leq \lambda'_s} -\nu (s,X_{s-})\boldsymbol{1}_{\theta \leq \lambda_s} \right) \Pi (\d s, \d \theta ,\d y)\right|\right]\\
    \leq&L_\lambda\E \left [\sup_{t\in[0,T]} \left|X'_t-X_t \right|\right] \\&+ L\E \left[\sup_{t\in [0,T]} \int_0^{t-} \int_{\R_+ \times \R}  \left|\phi(t-s)b(y)\left(\nu (s,X'_{s-})\boldsymbol{1}_{\theta \leq \lambda'_s} -\nu (s,X_{s-})\boldsymbol{1}_{\theta \leq \lambda_s} \right) \right|\Pi (\d s, \d \theta ,\d y)\right],\\
\end{align*}
where $L$ is the Lipschitz constant of $\psi$. Adding and subtracting $\nu(s,X_{s}')\boldsymbol{1}_{\theta \leq \lambda_s}$ in the last inequality we have
\begin{align*}
    \E \left[\sup_{t\in [0,T]}|\kappa'_t-\kappa_t|\right] \leq &L_\lambda\E \left [\sup_{t\in[0,T]} \left|X'_t-X_t \right|\right] \\&+ L\E \left[\sup_{t\in [0,T]} \int_0^{t-} \int_{\R_+ \times \R}  \left|\phi(t-s)b(y)\nu (s,X'_{s-})\left(\boldsymbol{1}_{\theta \leq \lambda'_s} -\boldsymbol{1}_{\theta \leq \lambda_s} \right) \right|\Pi (\d s, \d \theta ,\d y)\right]\\
    &+L\E \left[\sup_{t\in [0,T]} \int_0^{t-} \int_{\R_+ \times \R}  \left|\phi(t-s)b(y)\left(\nu (s,X'_{s-}) -\nu (s,X_{s-})\right)\boldsymbol{1}_{\theta \leq \lambda_s}  \right|\Pi (\d s, \d \theta ,\d y)\right].\\
    \end{align*}
    And since $\phi$ is assumed to be bounded,
    \begin{align*}     
   \E \left[\sup_{t\in [0,T]}|\kappa'_t-\kappa_t|\right]  \leq & L_\lambda \E \left [\sup_{t\in[0,T]} \left|X'_t-X_t \right|\right] \\&+ L\E \left[ \int_0^{T} \int_{\R_+ \times \R}  \left| \left \|\phi\right\|_\infty
b(y)\nu (s,X'_{s-})\left(\boldsymbol{1}_{\theta \leq \lambda'_s} -\boldsymbol{1}_{\theta \leq \lambda_s} \right) \right|\Pi (\d s, \d \theta ,\d y)\right]\\
    &+L\E \left[\int_0^{T} \int_{\R_+ \times \R}  \left| \|\phi\|_{\infty} b(y)\left(\nu (s,X'_{s-}) -\nu (s,X_{s-})\right)\boldsymbol{1}_{\theta \leq \lambda_s}  \right|\Pi (\d s, \d \theta ,\d y)\right].\\
\end{align*}
Since $\d s \d \theta m(\d y)$ is the compensator is $\Pi(\d s, \d \theta, \d y)$ and $\nu$ is bounded, we have that
\begin{align*}
    \E \left[\sup_{t\in [0,T]}|\kappa'_t-\kappa_t|\right] \leq & L_\lambda\E \left [\sup_{t\in[0,T]} \left|X'_t-X_t \right|\right] \\ &+ L\E \left[ \int_0^{T} \int_{\R_+ \times \R}  \left|\|\phi\|_\infty b(y)\nu (s,X'_{s-})\left(\boldsymbol{1}_{\theta \leq \lambda'_s} -\boldsymbol{1}_{\theta \leq \lambda_s} \right) \right|\d s \d \theta m(\d y)\right]\\
    &+L\E \left[\int_0^{T} \int_{\R_+ \times \R}  \left| \|\phi\|_\infty b(y)\left(\nu (s,X'_{s-}) -\nu (s,X_{s-})\right)\boldsymbol{1}_{\theta \leq \lambda_s}  \right|\d s \d \theta m(\d y)\right]\\
    \leq &  L_\lambda\E \left [\sup_{t\in[0,T]} \left|X'_t-X_t \right|\right] + L\E[b(Y)]\E \left[ \int_0^{T-}   \left|\|\phi\|_\infty \nu (s,X'_{s-}) \right|\left| \lambda'_s - \lambda_s \right|\d s\right]\\
    &+L\E[b(Y)]\E \left[\int_0^{T-} \left|\|\phi\|_\infty \left(\nu (s,X'_{s-}) -\nu (s,X_{s-})\right)\lambda_s \right|\d s\right]\\
    \leq & L_\lambda\E \left [\sup_{t\in[0,T]} \left|X'_t-X_t \right|\right]+C \E\left( \int_0^{T}   \left|\|\phi\|_{\infty}\right|\left| \lambda'_s - \lambda_s \right| \d s+\int_0^{T} \left|\|\phi\|_{\infty}\left(\nu (s,X'_{s}) -\nu (s,X_{s})\right)\lambda_s \right|\d s\right),
\end{align*}
where $C$ is a non-negative constant. Since Assumption \ref{ass:gronwall} is in force, the second integral in the last inequality is either 
\begin{enumerate}
    \item Bounded from above by $C\int_0^T \|\phi\|_{\infty}|X'_s-X_s|\d s$ if the jump rate $\psi$ is bounded. 
    \item Exactly equal to zero if $\nu$ does not depend on $x$.
\end{enumerate}
Hence, we have that 
\begin{equation}
    \label{ineq:kappa}
    \E \left[\sup_{t\in [0,T]}|\kappa'_t-\kappa_t|\right] \leq   \left(CT\E\left[\sup_{t\in [0,T]}|\lambda'_t-\lambda_t|\right]+(L_\lambda +CT)\E \left[\sup_{t\in [0,T]}|X'_t-X_t|\right]\right).
\end{equation}
We now examine the first component $Z$. Using the triangle inequality, the Bukholder-Davis-Gundy inequality and the Lipschitz property of the coefficients we have that
 \begin{align*}
     \E \left[\sup_{t\in[0,T]} |Z'_t-Z_t|\right]\leq &  \E\left[\sup_{t\in [0,T]}\left|\int_0^t \mu(s,X'_s) -\mu(s,X_s) \d s \right|\right]+\E\left[\sup_{t\in [0,T]}\left|\int_0^t \sigma(s,X'_s) -\sigma(s,X_s) \d W_s \right|\right]\\
     &+\E \left[\sup_{t\in [0,T]}\left|\int_0^t \int_{\R_+ \times \R} y\left(\gamma(s,X'_{s-}) \boldsymbol{1}_{\theta \leq \lambda'_s}-\gamma(s,X_{s-}) \boldsymbol{1}_{\theta \leq \lambda_s} \right)\Pi (\d s, \d \theta, \d y) \right|\right]\\
     \leq & C \left(\E\left[\int_0^T |X'_s-X_s|\d s\right] + \E \left[\left(\int_0^T|X'_s-X_s|^2 \d s\right)^{1/2}\right]\right)\\
     &+ C \left(\E \left[\int_0^T |\gamma(s,X_s)||\lambda'_s-\lambda_s|\d s\right] + \E \left[\int_0^T |\gamma(s,X'_s)-\gamma(s,X_s)|\lambda'_s \d s\right]\right) .
 \end{align*}
 Since Assumption \ref{ass:gronwall} is in force, the last integral is either exactly equal to zero or is bounded from above by $\int_0^T |X'_s-X_s|\d s$, and given that $\gamma$ is bounded we have that 
 \begin{equation}
     \label{ineq:Z}
     \E \left[\sup_{t\in [0,T]}|Z'_t-Z_t|\right]\leq C (T+\sqrt T)\left(\E\left[\sup_{t\in [0,T]}|\lambda'_t-\lambda_t|\right]+\E \left[\sup_{t\in [0,T]}|X'_t-X_t|\right]\right).
 \end{equation}
 Since $L_\lambda<1$, by taking $T$ small enough, Inequalities \eqref{ineq:kappa} and \eqref{ineq:Z} yield that $f$ is a contraction for the $\mathbb L^1$ norm $$\|(X,\lambda)\|:=\E\left[\sup_{t\in [0,T]}|\lambda_t|\right]+\E\left[\sup_{t\in [0,T]}|X_t|\right],$$
 and thus $f$ has a unique fixed point which solves SDE \eqref{eq:EDS} on $[0,T]$.\\
 Now that $\lambda$ and $X$ are constructed on $[0,T]$, we use them as "initial values" for the construction on $[T,2T]$. That is, $\lambda'_t=\lambda_t$ and $X'_t=X_t$ on $[0,T]$. 
\begin{align*}
    \E \left[\sup_{t\in [T,2T]}|\kappa'_t-\kappa_t|\right] \leq &L_\lambda\E \left [\sup_{t\in[T,2T]} \left|X'_t-X_t \right|\right] \\&+ L\E \left[\sup_{t\in [T,2T]} \int_0^{t-} \int_{\R_+ \times \R}  \left|\phi(t-s)b(y)\nu (s,X'_{s-})\left(\boldsymbol{1}_{\theta \leq \lambda'_s} -\boldsymbol{1}_{\theta \leq \lambda_s} \right) \right|\Pi (\d s, \d \theta ,\d y)\right]\\
    &+L\E \left[\sup_{t\in [T,2 T]} \int_0^{t-} \int_{\R_+ \times \R}  \left|\phi(t-s)b(y)\left(\nu (s,X'_{s-}) -\nu (s,X_{s-})\right)\boldsymbol{1}_{\theta \leq \lambda_s}  \right|\Pi (\d s, \d \theta ,\d y)\right]\\
    \leq & L_\lambda \E \left [\sup_{t\in[T,2T]} \left|X'_t-X_t \right|\right] \\&+ L\E \left[ \int_0^{2T} \int_{\R_+ \times \R}  \left| \left \|\phi\right\|_\infty
b(y)\nu (s,X'_{s-})\left(\boldsymbol{1}_{\theta \leq \lambda'_s} -\boldsymbol{1}_{\theta \leq \lambda_s} \right) \right|\Pi (\d s, \d \theta ,\d y)\right]\\
    &+L\E \left[\int_0^{2T} \int_{\R_+ \times \R}  \left| \|\phi\|_{\infty} b(y)\left(\nu (s,X'_{s-}) -\nu (s,X_{s-})\right)\boldsymbol{1}_{\theta \leq \lambda_s}  \right|\Pi (\d s, \d \theta ,\d y)\right].\\
\end{align*}
Given that the processes $X$ and $X'$, and $\lambda$ and $\lambda'$ coincide over $[0,T]$ we have that 
\begin{align*}
     \E \left[\sup_{t\in [T,2T]}|\kappa'_t-\kappa_t|\right] \leq &L_\lambda \E \left [\sup_{t\in[T,2T]} \left|X'_t-X_t \right|\right] \\&+ L\E \left[ \int_T^{2T} \int_{\R_+ \times \R}  \left| \left \|\phi\right\|_\infty
b(y)\nu (s,X'_{s-})\left(\boldsymbol{1}_{\theta \leq \lambda'_s} -\boldsymbol{1}_{\theta \leq \lambda_s} \right) \right|\Pi (\d s, \d \theta ,\d y)\right]\\
    &+L\E \left[\int_T^{2T} \int_{\R_+ \times \R}  \left| \|\phi\|_{\infty} b(y)\left(\nu (s,X'_{s-}) -\nu (s,X_{s-})\right)\boldsymbol{1}_{\theta \leq \lambda_s}  \right|\Pi (\d s, \d \theta ,\d y)\right].\\
\end{align*}

Like we did for $[0,T]$, $\d s \d \theta m(\d y)$ is the compensator is $\Pi(\d s, \d \theta, \d y)$ and $\nu$ is bounded, we thus have that
 \begin{align*}
     \E \left[\sup_{t\in [T,2T]}|\kappa'_t-\kappa_t|\right] \leq &L_\lambda \E \left [\sup_{t\in[T,2T]} \left|X'_t-X_t \right|\right] \\
     &+ L\E [b(Y)]\left \|\phi\right\|_\infty\E \left[ \int_T^{2T}   
 \left| \lambda'_s - \lambda_s \right|\d s\right]\\
    &+L\E [b(Y)] \left \|\phi\right\|_\infty\E \left[\int_T^{2T}  \left| \left(\nu (s,X'_{s-}) -\nu (s,X_{s-})\right)\boldsymbol{1}_{\theta \leq \lambda_s}  \right|\d s\right],\\
    \end{align*}
    and therefore,
    \begin{align*}
   \E \left[\sup_{t\in [T,2T]}|\kappa'_t-\kappa_t|\right]  \leq &\left(CT\E\left[\sup_{t\in [T,2T]}|\lambda'_t-\lambda_t|\right]+(L_\lambda +CT)\E \left[\sup_{t\in [T,2T]}|X'_t-X_t|\right]\right).
 \end{align*}
 
Which yields the existence and uniqueness on $[T,2T]$. The construction on any interval of finite length follows immediately. 
 
 For the estimate on the $p-$th moment of $X$, introduce $\tau_M := \inf\{ t \ge 0 : |X_t| \ge M \}$, for $M \ge 0$, and denote $X^M := X_{\cdot \wedge \tau_M}$. We have that 
 \begin{align*}
     X_t^M =& x_0 + \int_0^{t \wedge \tau_M} \mu(s,X_s^M) \d s + \int_0 ^{t \wedge \tau_M} \sigma(s,X_s^M) \d W_s + \int_0^{t \wedge \tau_M} \int_{\R_+ \times \R} y\gamma(s,X_{s-}^M) \boldsymbol{1}_{\theta \leq \lambda_s} \Pi (\d s, \d \theta, \d y)\\
     =&x_0 + \int_0^{t \wedge \tau_M} \mu(s,X_s^M)+\E[Y]\gamma(s,X_s^M)\lambda_s \d s + \int_0 ^{t \wedge \tau_M} \sigma(s,X_s^M) \d W_s \\ &+ \int_0^{t \wedge \tau_M} \int_{\R_+ \times \R} y\gamma(s,X_{s-}^M) \boldsymbol{1}_{\theta \leq \lambda_s} \bar \Pi (\d s, \d \theta, \d y)\\
 \end{align*}
 where $\bar \Pi(\d s, \d \theta, \d y):=\Pi(\d s, \d \theta, \d y)-\d s \d \theta m(\d y) $ is the compensated Poisson measure. By using Hölder and Burkholder-Davis-Gundy inequalities (\textit{cf.} Theorem 4.4.23 in \cite{Applebaum}) as well as the linear growth of the coefficients with respect to $x$, we have that 
 \begin{align*}       \E\left[\sup_{s\in [0,t]} | X^M_s |^p \right] 
       \leq &  C_T\left(1 + \E\left[\int_0^t \sup_{r\in[0,s]}|X_r^M|^p \d s \right] +\E\left[ \int_0^T\lambda_s^p \d s+\int_0^T\lambda_s \d s\right]\right), 
 \end{align*}
 where $C_T$ is a positive constant that does not depend on $M$.
 We now by Lemma \ref{lmm:higher_moments_lambda} that $\E[\sup_{t \in [0,T]} \lambda_t^p ] < \infty.$ Then, by Gr\"onwall's Lemma, there exist two constants $c_1$ and $c_2$, independent of $M$, such that:
 $$\E\left[\sup_{t\in[0,T]} | X_t^M |^p \right]\leq C_T. $$
 We obtain the desired result by applying the monotone convergence theorem.

\subsection{Proof of Theorem \ref{thm:continuity_X_noX}}
    \label{sec:thmcontinuity_X_noX}
    
    By compensating the Poisson measure we can rewrite $X$ and $\tilde X$ under the form
    \begin{equation*}
    \begin{cases}
        X_t=& x_0 + \int_0^t \mu(s, X_s)+\E b(Y)\lambda_s \gamma(s,X_s) \d s + \int_0 ^t \sigma(s, X_s) \d W_s + \int_0^t \int_{\R_+ \times \R} y\gamma(s, X_{s-}) \boldsymbol{1}_{\theta \leq \lambda_s} \bar\Pi (\d s, \d \theta, \d y)\\
        \tilde X_t=& x_0 + \int_0^t \mu(s,\tilde X_s) +\E b(Y) \tilde\lambda_s \gamma(s,\tilde X_s)\d s + \int_0 ^t \sigma(s,\tilde X_s) \d W_s + \int_0^t \int_{\R_+ \times \R} y\gamma(s,\tilde X_{s-}) \boldsymbol{1}_{\theta \leq \tilde \lambda_s} \bar\Pi (\d s, \d \theta, \d y)\\
    \end{cases}
    \end{equation*}
    which by taking the absolute value of the difference yields
    \begin{align*}
        \left | \tilde X_t-X_t \right|\leq &\int_0^t\left |\mu(s,\tilde X_s) -\mu(s,X_s)\right| \d s + \E b(Y)\int_0^t\left| \tilde \lambda_s \gamma(s,\tilde X_s)-\lambda_s\gamma(s, X_s)\right| \d s + \left |\int_0^t \left(\sigma(s,\tilde X_s)-\sigma(s, X_s)\right)\d W_s\right|\\
        &+\left|\int_0^t \int_{\R_+ \times \R}y\left(\gamma(s, \tilde X_{s-})\boldsymbol{1}_{\theta \leq \tilde \lambda_s}-\gamma(s,X_s)\boldsymbol{1}_{\theta \leq \lambda_s}\right) \bar \Pi (\d s, \d \theta, \d y) \right |\\
        \leq & C\left(\int_0^t|\tilde X_s-X_s|\d s +\int_0^t \tilde \lambda_s |\gamma(s,\tilde X_s)-\gamma(s,X_s)|\d s +\int_0^t|\gamma(s,X_s)||\tilde \lambda_s-\lambda_s| \d s\right)\\
        &+\left |\int_0^t \left(\sigma(s,\tilde X_s)-\sigma(s, X_s)\right)\d W_s\right| + \left|\int_0^t \int_{\R_+ \times \R}y\left(\gamma(s, \tilde X_{s-})\boldsymbol{1}_{\theta \leq \tilde \lambda_s}-\gamma(s,X_s)\boldsymbol{1}_{\theta \leq \lambda_s}\right) \bar \Pi (\d s, \d \theta, \d y) \right |
    \end{align*}
    and since Assumptions \ref{ass:gronwall} and \ref{ass:Lipschitz} hold, then
    \begin{align*}
        \left | \tilde X_t-X_t \right|\leq &C\left(\int_0^t|\tilde X_s-X_s|\d s  +\int_0^t|\tilde \lambda_s-\lambda_s| \d s\right)+\left |\int_0^t \left(\sigma(s,\tilde X_s)-\sigma(s, X_s)\right)\d W_s\right| \\
        &+ \left|\int_0^t \int_{\R_+ \times \R}y\left(\gamma(s, \tilde X_{s-})\boldsymbol{1}_{\theta \leq \tilde \lambda_s}-\gamma(s,X_s)\boldsymbol{1}_{\theta \leq \lambda_s}\right) \bar \Pi (\d s, \d \theta, \d y) \right |
    \end{align*}
    Using the Burkholder-Davis-Gundy inequality we have that 
    \begin{align*}
        \E \left[\sup_{s\in[0,t]}\left|\int_0^s \left(\sigma(s,\tilde X_s)-\sigma(s, X_s)\right)\d W_s\right |^2\right] & \leq C \E \left[\int_0^t\left(\sigma(s,\tilde X_s)-\sigma(s, X_s)\right)^2 \d s\right] \leq C \E \left[\int_0^t\left(\tilde X_s- X_s\right)^2 \d s\right]\\
        & \leq C \int_0^t \E \left[\sup_{u\in[0,s]}\left|\tilde X_u-X_u \right|^2\right] \d s,
    \end{align*}
    and 
    \begin{align*}
        \E &\left[\sup_{s\in[0,t]}\left|\int_0^t \int_{\R_+ \times \R}y\left(\gamma(s, \tilde X_{s-})\boldsymbol{1}_{\theta \leq \tilde \lambda_s}-\gamma(s,X_s)\boldsymbol{1}_{\theta \leq \lambda_s}\right) \bar \Pi (\d s, \d \theta, \d y) \right |^2\right]\\  &\leq  C \E \left[\int_0^t\int_{\R_+\times \R}\left| y\left(\gamma(s, \tilde X_{s-})\boldsymbol{1}_{\theta \leq \tilde \lambda_s}-\gamma(s,X_s)\boldsymbol{1}_{\theta \leq \lambda_s}\right)\right|^2 \d s \d \theta m(\d y)\right]\\
        &\leq C \E \left[\int_0^t \int_{\R_+}\left(\gamma(s,\tilde X_s)-\gamma(s,X_s)\right)^2 \boldsymbol{1}_{\theta \leq \tilde \lambda_s} \d \theta \d s\right]+ C \E \left[ \int_0^t \int_{\R_+}\gamma(s,X_s)^2\left|\boldsymbol{1}_{\theta \leq \tilde \lambda_s} -\boldsymbol{1}_{\theta \leq \lambda_s}\right| \d \theta \d s\right]\\
        &\leq C \E \left[\int_0^t \left(\gamma(s,\tilde X_s)-\gamma(s,X_s)\right)^2 \tilde \lambda_s  \d s\right]+ C \E \left[ \int_0^t\left|\tilde \lambda_s -\lambda_s\right| \d s\right],\\
    \end{align*}
    and since Assumption \ref{ass:gronwall} is in force 
    \begin{align*}
         \E &\left[\sup_{s\in[0,t]}\left|\int_0^t \int_{\R_+ \times \R}y\left(\gamma(s, \tilde X_{s-})\boldsymbol{1}_{\theta \leq \tilde \lambda_s}-\gamma(s,X_s)\boldsymbol{1}_{\theta \leq \lambda_s}\right) \bar \Pi (\d s, \d \theta, \d y) \right |^2\right]\\ 
         &\leq  C \int_0^t \E \left[ \sup_{u\in[0,s]} \left |\tilde X_u-X_u\right |^2  \right]  \d s+ C  \int_0^t \E \left[\left|\tilde \lambda_s -\lambda_s\right| \right]\d s.\\
    \end{align*}
    Hence, using the Cauchy-Schwarz inequality 
    \begin{align*}
        \E \left[\sup_{s\in [0,t]}\left | \tilde X_s-X_s\right|^2\right] \leq & C T\left( \int_0^t \E \left[\sup_{u\in [0,s]}\left | \tilde X_u-X_u\right|^2\right]\d s +\int_0^t\E \left[|\tilde \lambda_s-\lambda_s|^2\right] \d s + \int_0^t\E \left[|\tilde \lambda_s -\lambda_s|\right] \d s\right)
    \end{align*}
    and we get the result using Gr\"onwall's inequality coupled with Proposition  \ref{prop:continuity_lambda_noX}.
    
\subsection{Proof of Theorem \ref{thm:continuity_X}}
\label{sec:thmcontinuity_X}
We start with the upper bounds on the first moment. For any $0\leq u \leq s \leq t \leq T$ we have that 
\begin{equation*}
    \begin{cases}
        X_s&=X_u+\int_u^s\mu(r,X_r) \d r +\int_u^s \sigma(r,X_r)\d W_r+\int_u^s \int_{\R_+ \times \R} y\gamma(r, X_{r-}) \boldsymbol{1}_{\theta \leq \lambda_r} \Pi (\d r, \d \theta, \d y)\\
        \tilde X_s&=\tilde X_u+\int_u^s\mu(r,\tilde X_r) \d r +\int_u^s \sigma(r,\tilde X_r)\d W_r+\int_u^s \int_{\R_+ \times \R} y\gamma(r, \tilde X_{r-}) \boldsymbol{1}_{\theta \leq \tilde \lambda_r} \Pi (\d r, \d \theta, \d y)\\
    \end{cases}
\end{equation*}
Taking the absolute value of the difference yields
\begin{align*}
    \left |\tilde X_s -X_s \right | \leq & |\tilde X_u -X_u|+C\int_u^s |\tilde X_r-X_r| \d r + \left|\int_u^s(\sigma(r,\tilde X_r)-\sigma(r,X_r))\d W_r \right|\\
    &+\int_u^s\int_{\R_+\times \R}|y|\boldsymbol{1}_{\theta \leq \lambda_r}\left |\gamma(r,\tilde X_r)-\gamma(r,X_r)\right | \Pi(\d r, \d \theta,\d y)\\
    &+\int_u^s\int_{\R_+\times \R}|y|\left |\boldsymbol{1}_{\theta \leq \tilde \lambda_r}-\boldsymbol{1}_{\theta \leq \lambda_r}\right | \Pi(\d r, \d \theta,\d y).
\end{align*}
the Burkholder-Davis-Gundy inequality yields 
\begin{align*}
    \E \left [\sup_{s\in[u,t]}\left | \int_u^s \left(\sigma(r,\tilde X_r)-\sigma(r,X_r)\right)\d W_r\right | \right]\leq & C\E \left[ \left(\int_u^t \left(\sigma(r,\tilde X_r)-\sigma(r,X_r)\right)^2 \d r \right)^{1/2}\right]\\
    \leq &C\E \left[ \left(\int_u^t \left(\tilde X_r-X_r\right)^2 \d r \right)^{1/2}\right]\\
    \leq & C \sqrt{t-u} \E \left[\sup_{s\in[u,t]}|\tilde X_s -X_s|\right].
\end{align*}
Thus
\begin{align*}
    \E \left[\sup_{s\in[u,t]}\left | \tilde X_s-X_s\right|\right]\leq & \E |\tilde X_u-X_u | + C\left((t-u)+\sqrt{t-u}\right)\E \left[\sup_{s\in[u,t]}\left|\tilde X_s-X_s \right| \right]\\
    &+\int_u^t \E |Y| \E\left[\lambda_s\left| \gamma(s,\tilde X_s)-\gamma(s,X_s)\right|\right] \d s + \int_u^t\E |Y| \E \left[\left| \tilde \lambda_s-\lambda_s\right|\right] \d s \\
    \leq & \E|\tilde X_u-X_u | + C\left((t-u)+\sqrt{t-u}\right)\E \left[\sup_{s\in[u,t]}\left|\tilde X_s-X_s \right| \right]+ C\int_u^t \E \left[\left| \tilde \lambda_s-\lambda_s\right|\right] \d s \\
\end{align*}
where the last inequality is true because Assumption \ref{ass:gronwall} is in force. Using Proposition \ref{prop:continuity_lambda_depend_X}, we have that
\begin{align*}
     \E \left[\sup_{s\in[u,t]}\left | \tilde X_s-X_s\right|\right]\leq & \E \left[|\tilde X_u-X_u |\right] + C\left((t-u)+\sqrt{t-u}\right)\E \left[\sup_{s\in[u,t]}\left|\tilde X_s-X_s \right| \right]\\
     &+ C(t-u)\|\tilde \phi-\phi\|_1 +C(t-u)\E \left[\sup_{s\in[0,t]}\left |\tilde X _s -X_s \right |\right].
\end{align*}
We now assume that $t-u$ is smaller than one and we set $t-u=\delta :=\frac{1}{16C^2}$ to obtain
\begin{align*}
     \E \left[\sup_{s\in[u,t]}\left | \tilde X_s-X_s\right|\right]\leq & \E \left[|\tilde X_u-X_u |\right] +\frac{1}{2}\E \left[\sup_{s\in[u,t]}\left|\tilde X_s-X_s \right| \right]\\
     &+\sqrt \delta \|\tilde \phi-\phi\|_1 +\sqrt \delta \E \left[\sup_{s\in[0,t]}\left |\tilde X _s -X_s \right |\right],
\end{align*}
and thus 
\begin{equation}
\label{ineq:sup_L1}
     \E \left[\sup_{s\in[u,t]}\left | \tilde X_s-X_s\right|\right]\leq  2\E \left[|\tilde X_u-X_u |\right] + 2\sqrt \delta \|\tilde \phi-\phi\|_1 +2\sqrt \delta \E \left[\sup_{s\in[0,t]}\left |\tilde X _s -X_s \right |\right].
\end{equation}
We not set $f_n:=\E \left[\sup_{s \in [0,n\delta]}\left| \tilde X_s-X_s\right |\right]$ and bounding the maximum of two positive terms by their sum we have that 
\begin{align*}
    f_{n+1}=&\E \left[\sup_{s \in [0,(n+1)\delta]}\left| \tilde X_s-X_s\right |\right]
    =\E \left[\max\left(\sup_{s \in [0,n\delta]}\left| \tilde X_s-X_s\right |,\sup_{s \in [n\delta,(n+1)\delta]}\left| \tilde X_s-X_s\right | \right)\right]\\
    \leq &f_n+ \E \left[\sup_{s \in[n\delta,(n+1)\delta] } \left | \tilde X_s -X_s\right |\right].
\end{align*}
Using Inequality \eqref{ineq:sup_L1} and since without loss of generality one can assume that $C\geq 1$ (and hence $2 \sqrt \delta \leq \frac{1}{2}$) and we have that 
\begin{align*}
    f_{n+1}\leq & f_n + 2\E \left[|\tilde X_{n \delta}-X_{n\delta} |\right] + \frac{1}{2} \|\tilde \phi-\phi\|_1 +\frac{1}{2}\E \left[\sup_{s\in[0,(n+1)\delta]}\left |\tilde X _s -X_s \right |\right]\\
    = & f_n + 2\E \left[|\tilde X_{n \delta}-X_{n \delta} |\right] + \frac{1}{2} \|\tilde \phi-\phi\|_1 + \frac{1}{2}f_{n+1}\\
    \leq & 2 f_n +4 \E \left[|\tilde X_{n \delta}-X_{n \delta} |\right] +  \|\tilde \phi-\phi\|_1
\end{align*}
and using the obvious upper bound $ \E \left[|\tilde X_{n \delta}-X_{n \delta} |\right] \leq f_n$ we get that 
$$f_{n+1}\leq 6 f_n+\left \|\tilde \phi-\phi \right \|_1.$$
By multiplying both sides of the inequality by $6^{-(n+1)}$, taking the telescopic sum and keeping in mind that $f_0=0$, 
$$\E \left[\sup_{s \in [0,n\delta]} \left |\tilde X_s-X_s \right |\right]\leq \frac{6^n}{5}\|\tilde \phi-\phi\|_1.$$
Finally, the interval $[0,T]$ is divided into $n=\lfloor T/\delta\rfloor + 1 \leq (16TC^2+2)$ and the result immediately follows.

\subsection{Proof of Theorem \ref{thm:convergence-value-function}}
\label{sec:thmcvgvalue}
For notational simplicity, we assume without loss of generality that $f = 0$.\\
\textit{Step 1.} We first assume that $g$ is bounded. For $R > 0$, introduce $ B_R := \{ x \in \mathbb{R}^d : | x | \le R \}.$ Since $g$ is continuous on $\mathbb{R}^d$, it is uniformly continuous on $B_R$. Therefore, there exists a modulus of continuity $\rho_R$ such that:
$$ | g(x) - g(x') | \le \rho_R(| x-x'|) \ \mbox{for all} \ x, x' \in B_R. $$
Note that, since $g$ is bounded, we may assume without loss of generaility that $\rho_R$ is bounded by a constant independent of $R$. Therefore, given an arbitrary control $\alpha \in \cA$ controlling both $X$ and its Markov approximation $X^n$. Denote $\delta X_T^n := X_T^n - X_T$. We have the estimates:
\begin{align*}
\E\big[ | g(X_T^n) - g(X_T) | \big] \le& \E\big[ \rho_R\big(| \delta X_T^n  | \big) + | g(X_T^n) - g(X_T) | \boldsymbol{1}_{\{(X_T^n, X_T) \notin B_R^2 \} } \big] \\
\le& \E\big[ \rho_R\big(| \delta X_T^n  | \big) \big] + \Big( \E\big[ | g(X_T^n) - g(X_T) |^2 \big] \PP\big[ (X_T^n, X_T) \notin B_R^2 \big]\Big)^{1/2} \\
\le& \E\big[ \rho_R\big(| \delta X_T^n | \big) \big] + C \Big( \frac{\E\big[ |X_T^n| + |X_T| \big]}{R}\Big)^{1/2},
\end{align*}
where we successively used Cauchy-Schwarz inequality, the boundedness $g$ and Markov inequality. Then, by proposition \ref{prop:SDE}, we have for some nonnegative $C = C(|x_0|,T)$:
\begin{equation}\label{cv-ineq-1}
 \E\big[ | g(X_T^n) - g(X_T) | \big] \le \E\big[ \rho_R\big(| \delta X_T^n | \big) \big] + R^{-1/2}. 
 \end{equation}
 Now, observe that for any $\eta > 0$:
 \begin{align*}
     \E\big[ \rho_R(\delta X_T^n) \big] \le& \E\big[ \rho_R(\eta) + \rho_R(| \delta X_T^n |) \boldsymbol{1}_{\{ | \delta X_T^n | \ge \eta \}} ] \le \rho_R(\eta) + \frac{C}{\eta}\E\big[ | X_T^n - X_T | \big] \\
     \le& \rho_R(\eta) + \frac{\lVert \phi^n - \phi \rVert_1}{\eta},
 \end{align*}
 where we used the fact that $\rho_R$ is bounded independently of $R$, Markov inequality and Theorem \ref{thm:continuity_X}. Plugging this into \eqref{cv-ineq-1}, we obtain:
 $$  \E\big[ | g(X_T^n) - g(X_T) | \big] \le \rho_R(\eta) + \frac{\lVert \phi^n - \phi \rVert_1}{\eta} + R^{-1/2}. $$
 For an arbitrarily small $\varepsilon > 0$, we may then choose $R = R_\varepsilon$ and $\eta = \eta_\varepsilon$ such that
 $ \rho_R(\eta_\varepsilon) + R_\varepsilon^{-1/2} \le \frac{2}{3}\varepsilon, $
 and $n_\varepsilon$ such that:
 $ \frac{\lVert \phi^n - \phi \rVert_1}{\eta_\varepsilon} \le \frac{\varepsilon}{3} \ \mbox{for all} \ n \ge n_\varepsilon, $
 hence:
 $$ \E\big[ | g(X_T^n) - g(X_T) | \big] \le \varepsilon \quad \mbox{for all} \ n \ge n_\varepsilon. $$
 \textit{Step 2.} For $M > 0$, denote $g^M := g \wedge M \vee (-M)$. We then introduce:
 \begin{align*}
     V_0^M := \sup_{\alpha \in \cA} \E\big[ g^M(X_T) \big], \quad  V_0^{n,M} := \sup_{\alpha \in \cA} \E\big[ g^M(X_T^n) \big].
 \end{align*}
Observe that:
\begin{align*}
    | V_0^{M} - V_0 | \le& \sup_{\alpha \in \cA} \E\big[ | g^M(X_T) - g(X_T) | \big] = \sup_{\alpha \in \cA} \E\big[ | g^M(X_T) - g(X_T) | \boldsymbol{1}_{\{ | g(X_T) | \ge M \}} \big] \\
    \le& \sup_{\alpha \in \cA} C\Big( \E\big[ 1 + |X_T|^{2m} \big] \PP\big[ | g(X_T) | \ge M \big]  \Big)^{1/2} \le  \frac{C}{M} \sup_{\alpha \in \cA}  \E\big[ |g(X_T) | \big]^{1/2} \le \frac{C'}{M},
\end{align*}
where we used again Markov inequality, the polynomial growth of $g$ and \eqref{eq:indep-of-alpha}. Note that the constants $C$ and $C'$ are independent of $M$, $n$ and the choice of control $\alpha \in \cA$. Similarly, we have:
$$| V_0^{n,M} - V_0^n | \le \frac{C}{M}.$$
Finally, setting $M = M_{\varepsilon} = \frac{3C}{2\varepsilon}$, we have:
$ |V_0^n - V_0 | \le | V_0^{n,M_\varepsilon} - V_0^{M_\varepsilon} | + \frac{2}{3}\varepsilon.$
By Step 1, we know that $V_0^{n,M_\varepsilon} \to V_0^{M_\varepsilon}$ as $n \to \infty$. Therefore, there exists $n_\varepsilon$ such that:
$$|V_0^n - V_0 | \le \varepsilon \quad \mbox{for all} \ n \ge n_\varepsilon, $$
and this concludes the proof.

\appendix

\section{Technical lemmata}\label{sec:appendix}
\begin{Lemma}
\label{lmm:ineg_de_base}
Let $\phi:\R_+\to\mathcal M_{d,d}(\R_+)$ be such that $(\|\phi^{ij}\|_1)_{i,j=1,\cdots,d}$ has a spectral radius strictly less than one and $\boldsymbol g$ a locally integrable $d-$variate function.\\  
    Let $\boldsymbol f$ be a measurable, locally finite and non-negative $d-$variate function such that 
    $$\boldsymbol f(t) \leq \boldsymbol g(t) + \int_0^t \phi(t-s) \boldsymbol f(s) \d  s, \quad t \geq 0$$
    where the inequality holds component-wise.
    Then we have for all $t\geq 0$ $$\boldsymbol f (t) \leq \boldsymbol{g}(t) + \int_0^t Q_\phi (t-s) \boldsymbol{g}(s) \d s$$
    component-wise, with $Q_\phi=\sum_{n\geq 1}\phi^{(n)}$, where $\phi^{(n)}(t)=\int_0^t \phi(t-s)\phi^{(n-1)}(s)\d s$ and $\Phi^{(0)}=I_d$.
    In particular, if $\boldsymbol{g}(t) \equiv \boldsymbol{g}$ some constant vector, we have that
    $$\boldsymbol f (t) \leq \left( I-\|\Phi\|_1\right)^{-1}\boldsymbol{g}$$
\end{Lemma}
\begin{proof}

 We start by stating that since $(\|\phi^{ij}\|_1)_{i,j=1,\cdots,d}$ has a spectral radius strictly less than one, $Q_\phi$ is well defined (cf. \cite{MR3054533}).
Since $\boldsymbol{f}$ is bounded on $[0,t]$ we have that 
$$\int_0^t \phi^{(n)}(s) \boldsymbol f(t-s)\d s\leq \sup_{0\leq u\leq t}\boldsymbol{f}(u)\int_0^t\phi^{(n)}(s) \d s\xrightarrow[n\to+\infty]{}0.$$
Hence, by taking the convolution with $\phi^{(n)}$ and rearranging the terms we have 
\begin{align*}
    \int_0^t \phi^{(n)}(t-s)\boldsymbol{f}(s)\d s-\int_0^t \phi^{(n+1)}(t-s)\boldsymbol{f}(s) \d s &\leq \int_0^t \phi^{(n)}(t-s) \boldsymbol{g}(s) \d s\\
\end{align*}
which yields the result by telescoping. 
\end{proof}

\begin{Lemma}
\label{lmm:moments_lambda}
    Assume that in addition to Assumptions \ref{ass:gronwall} and \ref{ass:Lipschitz}, Assumption \ref{ass:stability} holds. Then there exist two positive constants $c_1$ and $c_2$ that do not depend on $T$ such that 
    $$\sup_{t\in [0,T]} \E [\lambda_t] \leq c_1 \quad \text{and} \quad \sup_{t\in [0,T]}\E[\lambda_t^2]\leq c_2.$$
\end{Lemma}
\begin{proof}
We note that if the first condition of Assumption \ref{ass:gronwall} holds, then there is nothing to prove. We thus focus on the second condition.
A complete proof of both of these bounds is present in \cite{coutin2024functionalapproximationmarkedhawkes}, we give here the proof of the first moment for the sake of completeness.
Using the Lipschitz growth of $\psi$ and the fact that $\lambda_\infty$ is bounded from above (by some constant $\bar \lambda$) we have that 
\begin{align*}
    \lambda_t \leq & \bar \lambda +\psi(0)+ L \left | \int_0^{t-} \int_{\R_+ \times \R}\phi (t-s) b(y) \nu (s,X_{s-})\boldsymbol{1}_{\theta \leq \lambda_s} \Pi (\d s, \d \theta ,\d y)\right|\\
     \leq & \bar \lambda +\psi(0) + L \int_0^{t-} \int_{\R_+ \times \R}\left |\phi (t-s) b(y)\boldsymbol{1}_{\theta \leq \lambda_s} \right|\Pi (\d s, \d \theta ,\d y),
\end{align*}
because $\nu$ is bounded from above by $1$. Taking the expected value of the last inequality yields
\begin{align*}
    \E [\lambda_t] \leq &  \bar \lambda + \psi(0)+ \lambda + L \E \left[\int_0^{t-} \int_{\R_+ \times \R}\left |\phi (t-s) b(y)\boldsymbol{1}_{\theta \leq \lambda_s} \right|\d s \d \theta m(\d y)\right]\\
    \leq & \bar \lambda +\psi(0)+ L \int_0^{t-} \left |\phi (t-s)\right| \E[b(Y)]\E \left[\lambda_s \right]\d s . \\
\end{align*}
The fact that Assumption \ref{ass:stability} is in force and using Lemma \ref{lmm:ineg_de_base} we obtain that 
\begin{equation}
    \label{ineq:lambda}
    \sup_{t\in [0,T]} \E[\lambda_t] \leq (1-L\E[b(Y)]\|\phi\|_1)^{-1}(\bar \lambda+\psi(0)).
\end{equation}
For the upper bound on the second moment, we compensate $\Pi$ which yields 
\begin{align*}
    \lambda_t \leq & \bar \lambda +\psi(0)+ L \int_0^{t-} \int_{\R_+ \times \R}\left |\phi (t-s) b(y)\boldsymbol{1}_{\theta \leq \lambda_s} \right|\left(\Pi (\d s, \d \theta ,\d y) -\d s \d \theta m(\d y)\right) + \int_{\R_+ \times \R}\left |\phi (t-s) \right| \E[b(Y)] \lambda_s  \d s\\
    \leq & M_t +  \int_{\R_+ \times \R}\left |\phi (t-s) \right| \E[b(Y)] \lambda_s  \d s
\end{align*}
where $M_t=\lambda +\psi(0)+ L \int_0^{t-} \int_{\R_+ \times \R}\left |\phi (t-s) b(y)\boldsymbol{1}_{\theta \leq \lambda_s} \right|\left(\Pi (\d s, \d \theta ,\d y) -\d s \d \theta m(\d y)\right)$. We then apply Lemma \ref{lmm:ineg_de_base} and proceed just like the proof of Lemma 6.4 in \cite{coutin2024functionalapproximationmarkedhawkes}.
Note that since we assumed that $\phi$ is bounded and $\mathbb L^1(\R_+)$, it is automatically in $\mathbb L^2(\R_+)$.
\end{proof}

\begin{Lemma}
    \label{lmm:higher_moments_lambda}
    Let $p>2$ and assume that Assumptions \ref{ass:gronwall} and \ref{ass:Lipschitz} hold. Assume furthermore that $b(Y)$ has a finite $p-$th moment. Then there exists a positive constant $C_{p}$ such that 
    $$\E \left[\sup_{t \in [0,T]} \lambda_t^p\right]\leq C_{p} e^{C_{p}T^{p}}.$$.
\end{Lemma}
\begin{proof}
In this proof we omit to explicitly write the localisation argument as it has already been used in the proof of Proposition \ref{prop:SDE}. 
    Using the Lipschitz growth of $\psi$ and the fact that $\lambda_\infty$ is bounded from above (by some constant $\bar \lambda$) we have that 
\begin{align*}
    \lambda_t \leq & \bar \lambda +\psi(0)+ L \left | \int_0^{t-} \int_{\R_+ \times \R}\phi (t-s) b(y) \nu (s,X_{s-})\boldsymbol{1}_{\theta \leq \lambda_s} \Pi (\d s, \d \theta ,\d y)\right|\\
     \leq & C \left(1 +  \int_0^{t-} \int_{\R_+ \times \R}\left | b(y) \right|\boldsymbol{1}_{\theta \leq \lambda_s}\Pi (\d s, \d \theta ,\d y) \right)\\
     \leq & C\left(1+ \int_0^{t-} \int_{\R_+ \times \R}\left | b(y) \right|\boldsymbol{1}_{\theta \leq \lambda_s} \left(\Pi (\d s, \d \theta ,\d y)-\d s \d \theta m(\d y)\right) + \int_0^t\lambda_s \d s\right),
\end{align*}
because $\nu$ and $\phi$ are bounded from above. Therefore we have that 
\begin{align*}
    \sup_{t \in [0,T]} \lambda_t^p \leq & C_p \left(1+  \sup_{t \in [0,T]}  \left(\int_0^{t-} \int_{\R_+ \times \R}\left | b(y) \right|\boldsymbol{1}_{\theta \leq \lambda_s} \left(\Pi (\d s, \d \theta ,\d y)-\d s \d \theta m(\d y)\right)\right)^p+ \sup_{t\in [0,T]}\left(\int_0^t \lambda_s \d s\right)^p\right) \\
    \leq & C_p \left(1+  \sup_{t \in [0,T]}  \left(\int_0^{t-} \int_{\R_+ \times \R}\left | b(y) \right|\boldsymbol{1}_{\theta \leq \lambda_s} \left(\Pi (\d s, \d \theta ,\d y)-\d s \d \theta m(\d y)\right)\right)^p+T^{p-1} \int_0^T\sup_{s\in [0,t]} \lambda_s^p \d t\right).\\
\end{align*}
Using the Burkholder-Davis-Gundy inequality for jump martingales (\textit{cf.} Theorem 4.4.23 in \cite{Applebaum}) and H\"older's inequality we have 
\begin{align*}
    \E \left[\sup_{t \in [0,T]} \lambda_t^p \right]\leq & C_p\left(1+\E \left[\left(\int_0^T\lambda_s \d s\right)^{p/2}\right]+T^{p-1}\int_0^T\E \left[\sup_{s\in [0,t]} \lambda_s^p\right]\d t\right)\\
    \leq & C_p\left(2+\E \left[\left(\int_0^T\lambda_s \d s\right)^{p}\right]+T^{p-1}\int_0^T\E \left[\sup_{s\in [0,t]} \lambda_s^p\right]\d t\right) \\
    \leq & C_p\left(1 + T^{p-1}\int_0^T\E \left[\sup_{s\in [0,t]} \lambda_s^p\right]\d t\right).
\end{align*}
Using Gr\"onwall's inequality, we have that $\E \left[\sup_{t \in [0,T]} \lambda_t^p \right]\leq C_p e^{C_pT^{p}}.$
\end{proof}
\section{Proofs of other results}
\label{sec:other_proofs}
\subsection{Proof of Proposition \ref{prop:continuity_lambda_noX}}
\label{sec:thmcontinuity_lambda_noX}
We start the proof by giving an upper bound on $\E \left[\left| \tilde \lambda_s-\lambda_s\right| \right]$.
    Using the Lipschitz property of $\psi$ coupled with the fact that $\nu$ is bounded by $1$, we have that 
\begin{align*}
     \left |\tilde \lambda_t-\lambda_t \right |\leq & L\left |\int_0^{t-}\int_{\R_+\times \R}  \nu (s) b(y)\left(\tilde\phi (t-s) \boldsymbol{1}_{\theta \leq \tilde \lambda_s}- \phi (t-s) \boldsymbol{1}_{\theta \leq \lambda_s} \right)\Pi (\d s, \d \theta ,\d y) \right|\\
        \leq &L\int_0^{t-}\int_{\R_+\times \R}b(y)\left |\left(\tilde\phi (t-s) \boldsymbol{1}_{\theta \leq \tilde \lambda_s}- \phi (t-s) \boldsymbol{1}_{\theta \leq \lambda_s} \right)\right|\Pi (\d s, \d \theta ,\d y), \\
\end{align*}
which, by adding and subtracting the adequate terms yields 
\begin{align*}
    \left |\tilde \lambda_t- \lambda_t \right| \leq 
         &L\int_0^{t-}\int_{\R_+\times \R}b(y) |\tilde \phi(t-s)|\left|\boldsymbol{1}_{\theta \leq \lambda_s}-\boldsymbol{1}_{\theta \leq \tilde \lambda_s } \right| \Pi (\d s, \d \theta ,\d y)\\
          &+L\int_0^{t-}\int_{\R_+\times \R} b(y)\left |\tilde \phi(t-s)-\phi(t-s)\right||\boldsymbol{1}_{\theta \leq \lambda_s}|\Pi (\d s, \d \theta ,\d y)\\
\end{align*}
and by taking the expected value we have 
\begin{align*}
    \E \left[ \left |\tilde \lambda_t- \lambda_t \right|\right] \leq &  \int_0^tL\E b(Y)\tilde \phi(t-s) \E \left| \tilde \lambda_s -\lambda_s\right| \d s+L \E b(Y)\int_0^t\left |\tilde \phi(t-s)-\phi(t-s)\right| \E\lambda_s \d s\\
    \leq & \int_0^tL\E b(Y)\tilde \phi(t-s) \E \left| \tilde \lambda_s -\lambda_s\right| \d s+L \E b(Y)\left \|\tilde \phi-\phi\right \| \frac{\psi(0)+\lambda_\infty}{1-L\E b(Y)\|\phi\|_1}  \\
\end{align*}
and by applying Lemma \ref{lmm:ineg_de_base}
\begin{align*}
   \E \left[ \left |\tilde \lambda_t- \lambda_t \right|\right] \leq C\|\tilde \phi- \phi\|_1.
\end{align*}

Now, we move over to the second moment, we start again by using the Lipschitz property of $\psi$
    \begin{align*}
        \left |\tilde \lambda_t-\lambda_t \right |\leq & L\left |\int_0^{t-}\int_{\R_+\times \R}b(y)\left(\tilde\phi (t-s) \boldsymbol{1}_{\theta \leq \tilde \lambda_s}- \phi (t-s) \boldsymbol{1}_{\theta \leq \lambda_s} \right)\Pi (\d s, \d \theta ,\d y) \right|\\
        \leq &L\int_0^{t-}\int_{\R_+\times \R}b(y)\left |\left(\tilde\phi (t-s) \boldsymbol{1}_{\theta \leq \tilde \lambda_s}- \phi (t-s)\boldsymbol{1}_{\theta \leq \lambda_s} \right)\right|\Pi (\d s, \d \theta ,\d y) \\
        \leq &L\int_0^{t-}\int_{\R_+\times \R}b(y)\left |\left(\tilde\phi (t-s) \boldsymbol{1}_{\theta \leq \tilde \lambda_s}- \phi (t-s) \boldsymbol{1}_{\theta \leq \lambda_s} \right)  \right|\bar\Pi (\d s, \d \theta ,\d y)\\
        &+L\int_0^{t-}\int_{\R_+\times \R} b(y)|\tilde \phi(t-s)| \left | \boldsymbol{1}_{\theta \leq \tilde \lambda_s}- \boldsymbol{1}_{\theta \leq  \lambda_s}\right| \d s \d \theta m(\d y)\\
        &+L\int_0^{t-}\int_{\R_+\times \R} b(y)\left |\tilde \phi(t-s)-\phi(t-s)\right|\boldsymbol{1}_{\theta \leq \lambda_s}\d s \d \theta m(\d y)\\
        =& L\E b(Y) \int_0^{t}|\tilde \phi (t-s)| \left| \tilde \lambda_s-\lambda_s\right |\d s +M_t + B_t
    \end{align*}
    where 
     \begin{align*}
        M_t&:= L\left |\int_0^{t-}\int_{\R_+\times \R}b(y)\left(\tilde\phi (t-s)\boldsymbol{1}_{\theta \leq \tilde \lambda_s}- \phi (t-s)\boldsymbol{1}_{\theta \leq \lambda_s} \right)\bar\Pi (\d s, \d \theta ,\d y) \right|\\
        B_t&:=L \E b(Y)\int_0^{t-}\left |\tilde \phi(t-s)-\phi(t-s)\right|\lambda_s\d s.
    \end{align*}
    Using Lemma \ref{lmm:ineg_de_base}, we have that 
    $$\left | \tilde \lambda_t-\lambda_t \right|\leq M_t  +B_t +\int_0^{t}Q(t-s)\left(M_s+B_s\right) \d s,$$
    where $Q=\sum_{n\geq 1} \left(\E b(Y)L|\tilde \phi|\right)^{(n)}$ is such that $\|Q\|_1=\frac{L\E b(Y)\| \tilde\phi\|_1}{1-L\E b(Y) \|\tilde \phi\|_1}$.\\
     We set $f(t,s)=Lb(y)\left|\tilde\phi (t-s)\boldsymbol{1}_{\theta \leq \tilde \lambda_s}- \phi (t-s)\boldsymbol{1}_{\theta \leq \lambda_s}\right|$ such that 
    $M_t=\int_0^{t-}\int_{\R\times \R_+}f(t,s) \bar \Pi(\d s,\d \theta,\d y)$.
    By adding and subtracting the adequate terms we have that
    Using Ito's Isometry, we have that 
    \begin{align*}
        \E \left[M_tM_u\right]&=\E \left[\int_0^{\min(t,u)}\int_{\R_+\times \R}\left(f(t,s)f(u,s)\right) \d s\d \theta m(\d y)\right]
    \end{align*}
    where the product $f(t,s)f(u,s)$ is bounded from above by 
    \begin{align*}
        f(t,s)f(u,s) \leq & \left|\tilde\phi (t-s)\tilde\phi (u-s)\right|\left|\boldsymbol{1}_{\theta \leq \tilde \lambda_s}- \boldsymbol{1}_{\theta \leq \lambda_s} \right|+\left|\tilde \phi(t-s) -\phi(t-s)\right|\left| \tilde \phi (u-s)-\phi(u-s)\right| \boldsymbol{1}_{\theta \leq \lambda_s}\\
        &+\left|\tilde \phi(u-s)\right|\left|\boldsymbol{1}_{\theta \leq \tilde \lambda_s}- \boldsymbol{1}_{\theta \leq \lambda_s} \right|\left|\tilde \phi(t-s) -\phi(t-s)\right|+\left|\tilde \phi(t-s)\right|\left|\boldsymbol{1}_{\theta \leq \tilde \lambda_s}- \boldsymbol{1}_{\theta \leq \lambda_s} \right|\left|\tilde \phi(u-s) -\phi(u-s)\right|.
    \end{align*}
    Assuming that $u\leq t$ we have that 
    \begin{align*}
        \E [M_uM_t] \leq &C \int_0^u \left|\tilde\phi (t-s)\tilde\phi (u-s)\right| \E\left| \tilde \lambda_s-  \lambda_s \right| \d s +\int_0^u\left|\tilde \phi(t-s) -\phi(t-s)\right|\left| \tilde \phi (u-s)-\phi(u-s)\right|  \E\lambda_s \d s\\
        &+\int_0^u\left|\tilde \phi(u-s)\right|\E\left|\tilde \lambda_s-  \lambda_s \right|\left|\tilde \phi(t-s) -\phi(t-s)\right|\d s+\int_0^u\left|\tilde \phi(t-s)\right|\E\left| \tilde \lambda_s- \lambda_s \right|\left|\tilde \phi(u-s) -\phi(u-s)\right| \d s\\
        \leq&C \int_0^u \left|\tilde\phi (u-s)\right| \E\left| \tilde \lambda_s-  \lambda_s \right| \d s+\left(\int_0^t\left|\tilde \phi(t-s) -\phi(t-s)\right|^2 \d s\right)^{1/2}\left(\int_0^u\left| \tilde \phi (u-s)-\phi(u-s)\right|^2  \d s\right)^{1/2}\\
        &+\int_0^u\left|\tilde \phi(u-s)\right|\E\left|\tilde \lambda_s-  \lambda_s \right| \d s+\int_0^u\E\left| \tilde \lambda_s- \lambda_s \right|\left|\tilde \phi(u-s) -\phi(u-s)\right| \d s\\
    \end{align*}
    and given that $\E |\tilde \lambda_s-\lambda_s| \leq C\|\tilde \phi-\phi\|_1$ (the intensity does not depend on $X$) then 
    $$\E [M_uM_t]\leq C \left(\|\tilde \phi-\phi\|_1+\|\tilde \phi-\phi\|^2_1 +\|\tilde \phi-\phi\|_2^2 \right).$$
    For the term $B$, we have that 
    \begin{align*}
        \E \left[B_t B_u\right] \leq & C \int_0^t\int_0^u \left|\tilde \phi(s)-\phi(s) \right| \left|\tilde \phi(v)-\phi(v) \right|\E \left[\lambda_{t-s}\lambda_{u-v}\right]\d s \d v\\
        \leq & C \int_0^t\int_0^u \left|\tilde \phi(s)-\phi(s) \right| \left|\tilde \phi(v)-\phi(v) \right|\E \left[\lambda_{t-s}^2\right]^{1/2} \left[\lambda_{u-v}^2\right]^{1/2}\d s \d v
        \leq C\|\tilde \phi-\phi\|_1^2
    \end{align*}
    which yields 
    \begin{align*}
        \E \left[\left | \tilde \lambda_t-\lambda_t\right |^2\right] \leq & C \left(\E [M^2_t]+\E[B^2_t] + \int_0^t \int_0^s Q(t-s)Q(t-u) \left(\E \left[M_sM_u\right]+\E[B_s B_u]\right)\d u \d s\right)\\
        \leq &C \left(\|\tilde \phi-\phi\|_1+\|\tilde \phi-\phi\|^2_1 +\|\tilde \phi-\phi\|_2^2 \right) \left(1+\left(\int_0^t Q(t-s)\d s\right)^2\right)
    \end{align*}

    \subsection{Proof of Proposition \ref{prop:continuity_lambda_depend_X}}
    \label{sec:propcontinuity_lambda_depend_X}
        Using the Lipschitz property of $\psi$ we have that 
\begin{align*}
     \left |\tilde \lambda_t-\lambda_t \right |\leq & \left |\lambda_{\infty}(t,\tilde X_{t-})-\lambda_\infty(t,X_{t-}) \right |\\&+L\left |\int_0^{t-}\int_{\R_+\times \R}b(y)\left(\tilde\phi (t-s) \nu (s,\tilde X_{s-})\boldsymbol{1}_{\theta \leq \tilde \lambda_s}- \phi (t-s) \nu (s,X_{s-})\boldsymbol{1}_{\theta \leq \lambda_s} \right)\Pi (\d s, \d \theta ,\d y) \right|\\
        \leq & \left |\lambda_{\infty}(t,\tilde X_{t-})-\lambda_\infty(t,X_{t-}) \right |\\ &+L\int_0^{t-}\int_{\R_+\times \R}b(y)\left |\left(\tilde\phi (t-s) \nu (s,\tilde X_{s-})\boldsymbol{1}_{\theta \leq \tilde \lambda_s}- \phi (t-s) \nu (s,X_{s-})\boldsymbol{1}_{\theta \leq \lambda_s} \right)\right|\Pi (\d s, \d \theta ,\d y), \\
\end{align*}
which, by adding and subtracting the adequate terms and bounding $\nu$ from above by 1 yields 
\begin{align*}
    \left |\tilde \lambda_t- \lambda_t \right| \leq &L\int_0^{t-}\int_{\R_+\times \R} b(y)|\tilde \phi(t-s)| \left |\nu(s,\tilde X_{s-})-\nu(s, X_{s-}) \right|  \boldsymbol{1}_{\theta \leq \tilde \lambda_s} \Pi (\d s, \d \theta ,\d y)\\
    &+\left |\lambda_{\infty}(t,\tilde X_{t-})-\lambda_\infty(t,X_{t-}) \right |+L\int_0^{t-}\int_{\R_+\times \R} b(y)\left |\tilde \phi(t-s)-\phi(t-s)\right||\boldsymbol{1}_{\theta \leq \lambda_s}|\Pi (\d s, \d \theta ,\d y)\\
         &+L\int_0^{t-}\int_{\R_+\times \R} b(y) |\tilde \phi(t-s)|\left|\boldsymbol{1}_{\theta \leq \lambda_s}-\boldsymbol{1}_{\theta \leq \tilde \lambda_s } \right| \Pi (\d s, \d \theta ,\d y)\\
\end{align*}
and by taking the expected value we have 
\begin{align*}
    \E \left[ \left |\tilde \lambda_t- \lambda_t \right|\right] \leq &L\E b(Y)\E \left[\int_0^t|\tilde \phi(t-s)|\left | \nu(s,\tilde X_s)-\nu(s,X_s)\right| \lambda_s \d s \right] + \E \left |\lambda_{\infty}(t,\tilde X_{t-})-\lambda_\infty(t,X_{t-}) \right | \\
    &+L \E b(Y)\int_0^t\left |\tilde \phi(t-s)-\phi(t-s)\right| \E\lambda_s \d s+\int_0^tL\E b(Y)\tilde \phi(t-s) \E \left| \tilde \lambda_s -\lambda_s\right| \d s\\
    \leq &\left(L\E b(Y)|\tilde \phi|*\E[|\nu(\cdot,\tilde X)-\nu(\cdot,X)|\lambda]\right) +  \E \left |\lambda_{\infty}(t,\tilde X_{t-})-\lambda_\infty(t,X_{t-}) \right |\\
    &+L \E b(Y) \frac{\bar\lambda_{\infty}+\psi(0)}{1-L \E b(Y)\|\phi\|_1} \|\tilde \phi -\phi\|_1 +\int_0^tL\E b(Y)\tilde \phi(t-s) \E \left| \tilde \lambda_s -\lambda_s\right| \d s\\
\end{align*}
and by applying Lemma \ref{lmm:ineg_de_base}

\begin{align*}
   \E \left[ \left |\tilde \lambda_t- \lambda_t \right|\right] \leq &C\bigg(\|\tilde \phi- \phi\|_1+\left( Q*\E\left [ \left|\lambda_\infty(\cdot,\tilde X)-\lambda_\infty(\cdot,X)\right| \right] \right)_t\\&+\left(Q*(L\E b(Y)\tilde \phi)*\E \left[|\nu(\cdot,\tilde X)-\nu(\cdot,X)|\lambda\right]\right)_t\bigg)\\
   &+\E\left [ \left|\lambda_\infty(t,\tilde X_t)-\lambda_\infty(t,X_t)\right| \right]+\left((L\E b(Y)\tilde \phi)*\E \left[|\nu(\cdot,\tilde X)-\nu(\cdot,X)|\lambda\right]\right)_t.
\end{align*}

where $Q= \sum_{n\geq 1}(\E b(Y)L \tilde \phi)^{(n)}$. Hence, $Q*(L\E b(Y)\tilde \phi) = \sum_{n\geq 2}(\E b(Y)L \tilde \phi)^{(n)}$ is a positive integrable function and the result follows. \\

\textbf{Funding}\\
 Mahmoud Khabou acknowledges support from EPSRC NeST Programme grant EP/X002195/1.

 \textbf{Conflict of interest}\\
There were no competing interests to declare which arose during the
preparation or publication process of this article.

{\footnotesize{
\printbibliography}}

\end{document}